\documentclass[11pt,reqno]{NumPDEsArticle}

\usepackage[english]{babel}
\usepackage{csquotes}
\usepackage{microtype}
\usepackage{accents}
\DeclareRobustCommand{\starred}[1]{\accentset{\star}{#1}}
\usepackage{enumitem}
\usepackage{amsmath, amssymb}
\usepackage{orcidlink}
\usepackage[basicdelimiters]{NumPDEsMacros}
\newcommand{\Eta}{{\sf H}}
\newcommand{\Zeta}{{\sf Z}}

\newcommand{\coarse}{H}
\newcommand{\fine}{h}

\newcommand{\kk}{{\underline{k}}}

\newcommand{\kmax}{{\underline{k}}}

\newcommand{\lmax}{{\underline{\ell}}}
\newcommand{\qctr}{q_{\textup{ctr}}}
\newcommand{\lctr}{\lambda}

\newcommand{\thetaopt}{\const{\theta}{opt}}
\newcommand{\thetamark}{\starred{\theta}}

\newcommand{\Caux}{\const{C}{aux}}

\newcommand{\ellu}{{\underline{\ell}}}
\newcommand{\Ccost}{\const{C}{cost}}
\newcommand{\Cstabeta}{\Cstab}
\newcommand{\Cstabzeta}{C_{\textup{stab}}^{\zeta}}
\newcommand{\Creleta}{\Crel}
\newcommand{\Crelzeta}{C_{\textup{rel}}^{\zeta}}
\newcommand{\Cdreleta}{\Cdrel}
\newcommand{\Cmoneta}{\Cmon}

\newcommand{\Ceqv}{\const{C}{eqv}}

\newcommand\f{\boldsymbol{f}}

\DeclareMathOperator*{\esssup}{\operatorname{ess\,sup}}
\DeclareMathOperator*{\essinf}{\operatorname{ess\,inf}}
\newcommand{\cost}{\operatorname{\textnormal{\texttt{cost}}}}

\newcommand{\myspace}{\;\,}
\newcommand{\eqpad}{\mathrel{\myspace = \myspace}}
\newcommand{\leqpad}{\mathrel{\myspace\leq\myspace}}
\newcommand{\lesssimpad}{\mathrel{\myspace\lesssim\myspace}}

\newcommand{\lepad}{\mathrel{\myspace < \myspace}}

\definecolor{TUBlue}{HTML}{006699}
\definecolor{TUGreen}{HTML}{007E71}
\definecolor{TUGrey}{HTML}{646363}
\definecolor{TUMagenta}{HTML}{BA4682}
\definecolor{TUYellow}{HTML}{E18922}
\definecolor{TULightGrey}{RGB}{238,238,238}
\definecolor{TURed}{RGB}{252,25,25}

\definecolor{hiBlue}{rgb}{0.0239,0.6287,0.8038}
\definecolor{hiYellow}{rgb}{0.9598,0.9218,0.0948}
\definecolor{hiGreen}{rgb}{0.6099,0.7473,0.4337}
\def\hb#1{{\colorbox{hiBlue}{#1}}}
\def\hy#1{{\colorbox{hiYellow}{#1}}}
\def\hg#1{{\colorbox{hiGreen}{#1}}}

\definecolor{psychedelicpurple}{rgb}{0.87, 0.0, 1.0}
\definecolor{cobalt}{rgb}{0, 0.28, 0.67}

\usepackage{graphicx}
\usepackage[labelformat=simple]{subcaption}

\usepackage{tikz}
\usetikzlibrary{calc,matrix}

\usepackage{pgfplots}
\pgfplotsset{compat=newest}
\usepackage{pgfplotstable}

\usepackage{tabularray}
\UseTblrLibrary{booktabs}
\SetTblrInner[booktabs]{
  colsep=4pt,
  leftsep=4pt,
  rightsep=4pt,
  abovesep = 0ex,
  belowsep = 0ex,
  rowsep = 0pt
}
\usepackage{diagbox}

\definecolor{pyBlue}{HTML}{1f77b4}
\definecolor{pyRed}{HTML}{d62728}
\definecolor{pyGreen}{HTML}{2ca02c}
\definecolor{pyOrange}{HTML}{ff7f0e}
\definecolor{pyPurple}{HTML}{9467bd}
\definecolor{pyYellow}{HTML}{bcbd22}
\definecolor{pyGrey}{HTML}{7f7f7f}
\definecolor{pyCyan}{HTML}{17becf}
\definecolor{pyBrown}{HTML}{8c564b}
\definecolor{pyPink}{HTML}{e377c2}

\colorlet{col1}{pyGrey}
\colorlet{col2}{pyBlue}
\colorlet{col3}{pyRed}
\colorlet{col4}{pyGreen}
\colorlet{col5}{pyOrange}
\colorlet{col6}{pyPurple}
\colorlet{col7}{pyCyan}
\colorlet{col8}{pyYellow}
\colorlet{col9}{pyBrown}
\colorlet{col10}{pyPink}

\definecolor{2col1}{HTML}{332288}
\definecolor{2col2}{HTML}{88CCEE}
\definecolor{2col3}{HTML}{44AA99}
\definecolor{2col4}{HTML}{117733}
\definecolor{2col5}{HTML}{999933}
\definecolor{2col6}{HTML}{DDCC77}
\definecolor{2col7}{HTML}{CC6677}
\definecolor{2col8}{HTML}{882255}
\definecolor{2col9}{HTML}{AA4499}

\pgfplotsset{%
    degdefault/.style = {%
        mark = *,%
        mark size = 2pt,%
        every mark/.append style = {solid},%
        gray,%
        every mark/.append style = {fill = gray!60!white}%
    },%
    p1L1/.style = {%
        degdefault,%
        mark = square*,%
        mark size = 1.66pt,%
        col1,%
        every mark/.append style = {fill = col1!40!white}%
    },%
    p2L1/.style = {%
        degdefault,%
        mark = triangle*,%
        mark size = 2.75pt,%
        col1,%
        every mark/.append style = {fill = col1!40!white}%
    },%
    p3L1/.style = {%
        degdefault,%
        mark = pentagon*,%
        mark size = 2.2pt,%
        col1,%
        every mark/.append style = {fill = col1!40!white}%
    },%
    p4L1/.style = {%
        degdefault,%
        col1,%
        every mark/.append style = {fill = col1!40!white}%
    },%
    p1L2/.style = {%
        degdefault,%
        col2,%
        every mark/.append style = {fill = col2!60!white}%
    },%
    p1L3/.style = {%
        degdefault,%
        mark = triangle*,%
        mark size = 2.75pt,%
        col3,%
        every mark/.append style = {fill = col3!60!white}%
    },%
    p2L2/.style = {%
        degdefault,%
        mark = square*,%
        mark size = 1.66pt,%
        col4,%
        every mark/.append style = {fill = col4!60!white}%
    },%
    p2L3/.style = {%
        degdefault,%
        mark = pentagon*,%
        mark size = 2.2pt,%
        col5,%
        every mark/.append style = {fill = col5!60!white}%
    },%
    p3L2/.style = {%
        degdefault,%
        mark = diamond*,%
        mark size = 2.75pt,%
        col6,%
        every mark/.append style = {fill = col6!60!white}%
    },%
    p3L3/.style = {%
        degdefault,%
        mark = halfcircle*,%
        every mark/.append style = {rotate = 315},%
        mark size = 2pt,%
        col7,%
        every mark/.append style = {fill = col7!60!white}%
    },%
    p4L2/.style = {%
        degdefault,%
        mark = halfsquare*,%
        every mark/.append style = {rotate = 135},%
        mark size = 2.35pt,%
        col8,%
        every mark/.append style = {fill = col8!60!white}%
    },%
    p4L3/.style = {%
        degdefault,%
        mark = halfdiamond*,%
        every mark/.append style = {rotate = 180},%
        mark size = 2.75pt,%
        col9,%
        every mark/.append style = {fill = col9!60!white}%
    },%
    direct/.style = {%
        dashed,%
        every mark/.append style = {%
            black!50!white,%
            fill = black!20!white
        }%
    },%
    reference/.style = {%
        dashed,%
    },%
    iterative/.style = {%
            solid%
    },%
}

\pgfplotsset{%
    colormap={parula}{%
        rgb=(0.2081,0.1663,0.5292)
        rgb=(0.2116,0.1898,0.5777)
        rgb=(0.2123,0.2138,0.627)
        rgb=(0.2081,0.2386,0.6771)
        rgb=(0.1959,0.2645,0.7279)
        rgb=(0.1707,0.2919,0.7792)
        rgb=(0.1253,0.3242,0.8303)
        rgb=(0.0591,0.3598,0.8683)
        rgb=(0.0117,0.3875,0.882)
        rgb=(0.006,0.4086,0.8828)
        rgb=(0.0165,0.4266,0.8786)
        rgb=(0.0329,0.443,0.872)
        rgb=(0.0498,0.4586,0.8641)
        rgb=(0.0629,0.4737,0.8554)
        rgb=(0.0723,0.4887,0.8467)
        rgb=(0.0779,0.504,0.8384)
        rgb=(0.0793,0.52,0.8312)
        rgb=(0.0749,0.5375,0.8263)
        rgb=(0.0641,0.557,0.824)
        rgb=(0.0488,0.5772,0.8228)
        rgb=(0.0343,0.5966,0.8199)
        rgb=(0.0265,0.6137,0.8135)
        rgb=(0.0239,0.6287,0.8038)
        rgb=(0.0231,0.6418,0.7913)
        rgb=(0.0228,0.6535,0.7768)
        rgb=(0.0267,0.6642,0.7607)
        rgb=(0.0384,0.6743,0.7436)
        rgb=(0.059,0.6838,0.7254)
        rgb=(0.0843,0.6928,0.7062)
        rgb=(0.1133,0.7015,0.6859)
        rgb=(0.1453,0.7098,0.6646)
        rgb=(0.1801,0.7177,0.6424)
        rgb=(0.2178,0.725,0.6193)
        rgb=(0.2586,0.7317,0.5954)
        rgb=(0.3022,0.7376,0.5712)
        rgb=(0.3482,0.7424,0.5473)
        rgb=(0.3953,0.7459,0.5244)
        rgb=(0.442,0.7481,0.5033)
        rgb=(0.4871,0.7491,0.484)
        rgb=(0.53,0.7491,0.4661)
        rgb=(0.5709,0.7485,0.4494)
        rgb=(0.6099,0.7473,0.4337)
        rgb=(0.6473,0.7456,0.4188)
        rgb=(0.6834,0.7435,0.4044)
        rgb=(0.7184,0.7411,0.3905)
        rgb=(0.7525,0.7384,0.3768)
        rgb=(0.7858,0.7356,0.3633)
        rgb=(0.8185,0.7327,0.3498)
        rgb=(0.8507,0.7299,0.336)
        rgb=(0.8824,0.7274,0.3217)
        rgb=(0.9139,0.7258,0.3063)
        rgb=(0.945,0.7261,0.2886)
        rgb=(0.9739,0.7314,0.2666)
        rgb=(0.9938,0.7455,0.2403)
        rgb=(0.999,0.7653,0.2164)
        rgb=(0.9955,0.7861,0.1967)
        rgb=(0.988,0.8066,0.1794)
        rgb=(0.9789,0.8271,0.1633)
        rgb=(0.9697,0.8481,0.1475)
        rgb=(0.9626,0.8705,0.1309)
        rgb=(0.9589,0.8949,0.1132)
        rgb=(0.9598,0.9218,0.0948)
        rgb=(0.9661,0.9514,0.0755)
        rgb=(0.9763,0.9831,0.0538)
    }
}

\newcommand\drawslopetriangle[4][ST]{
  \pgfplotsextra
  {%
    \pgfkeys{/pgf/fpu=true}%
    \pgfmathsetmacro\leftcoord{#3}%
    \pgfmathsetmacro\rightcoord{10*#3}%
    \pgfmathsetmacro\bottomcoord{#4}%
    \pgfmathsetmacro\topcoord{10^(#2)*#4}%
    \pgfkeys{/pgf/fpu=false}%

    \coordinate (#1-BL) at (axis cs:\leftcoord,\bottomcoord);
    \coordinate (#1-BR) at (axis cs:\rightcoord,\bottomcoord);
    \coordinate (#1-TL) at (axis cs:\leftcoord,\topcoord);

    \shadedraw[%
      bottom color = black!20,%
      middle color = black!5,%
      top color    = white,%
      draw         = black%
    ]
      (#1-TL) -- (#1-BL) node[midway, left=-2pt] {\scriptsize\(#2\)}
      -- (#1-BR) node[midway, below=-2pt] {\scriptsize\(1\)} -- (#1-TL);
  }
}
\newcommand\drawswappedslopetriangle[4][SST]{
  \pgfplotsextra
  {%
    \pgfkeys{/pgf/fpu=true}%
    \pgfmathsetmacro\leftcoord{#3/10}%
    \pgfmathsetmacro\rightcoord{#3}%
    \pgfmathsetmacro\topcoord{#4}%
    \pgfmathsetmacro\bottomcoord{10^(-#2)*#4}%
    \pgfkeys{/pgf/fpu=false}%

    \coordinate (#1-TR) at (axis cs:\rightcoord,\topcoord);
    \coordinate (#1-BR) at (axis cs:\rightcoord,\bottomcoord);
    \coordinate (#1-TL) at (axis cs:\leftcoord,\topcoord);

    \shadedraw[%
      bottom color = black!20,%
      middle color = black!5,%
      top color    = white,%
      draw         = black%
    ]
      (#1-BR) -- (#1-TR) node[midway, right=-2pt] {\scriptsize\(#2\)}
      -- (#1-TL) node[midway, above=-2pt] {\scriptsize\(1\)} -- (#1-BR);
  }
}

\newlength{\convergenceWidth}
\newlength{\meshWidth}
\title[Quasi-optimal complexity of iterative Galerkin methods]{Quasi-optimal complexity of iterative Galerkin methods\\ driven by an elliptic reconstruction error estimator}

\author{Maximilian Brunner~\orcidlink{0000-0003-0636-1491}}
\author{Gregor Gantner~\orcidlink{0000-0002-0324-5674}}
\author{Christoph Lietz~\orcidlink{0009-0005-1857-0954}}
\author{Dirk Praetorius~\orcidlink{0000-0002-1977-9830}}

\address{TU Wien, Institute of Analysis and Scientific Computing, Wiedner Hauptstr. 8--10/E101/4, 1040 Vienna, Austria}

\email{maximilian.brunner@asc.tuwien.ac.at}
\email{christoph.lietz@asc.tuwien.ac.at \quad {\normalfont{(corresponding author)}}}
\email{dirk.praetorius@asc.tuwien.ac.at}

\address{University of Twente, Department of Applied Mathematics, P.O.\ Box 217, 7500 AE Enschede, The Netherlands}
\email{gregor.gantner@utwente.nl}

\keywords{adaptive finite element method, quasilinear PDEs, elliptic reconstruction, a posteriori error estimation, iterative linearization, convergence analysis, optimal convergence rates}
\subjclass[2010]{65N30, 65N50, 65N15, 65Y20, 41A25}
\thanks{%
  This research was funded in whole or in part by the Austrian Science Fund (FWF)
  [\href{https://www.fwf.ac.at/en/research-radar/10.55776/I6802}{10.55776/I6802} and \href{https://www.fwf.ac.at/en/research-radar/10.55776/PAT3446525}{10.55776/PAT3446525}]
  and the German Research Foundation (DFG) under the individual research grant 545527047.
  Additionally, Christoph Lietz is supported by the Vienna School of Mathematics.
  For open access purposes, the authors have applied a CC BY public copyright license
  to any author accepted manuscript version arising from this submission.
}

\begin{document}

\maketitle
\thispagestyle{fancy}

\begin{abstract}
  We study an iterative Galerkin method for quasilinear elliptic problems in the Browder--Minty setting.
  The resulting discrete nonlinear systems are solved by linearization via a (damped) Zarantonello iteration.
  Unlike prior work, adaptive mesh refinement is driven by an elliptic reconstruction error estimator, which is natural in the sense that the \textsl{a~posteriori} bounds for the linearization and discretization errors are well separated.
  For this setting, we present the first comprehensive convergence analysis of the corresponding algorithm.
  We prove unconditional full R-linear convergence of a suitable quasi-error that combines linearization and discretization errors.
  For sufficiently small adaptivity parameters, we further establish optimal convergence rates with respect to the number of degrees of freedom and quasi-optimal complexity, i.e., optimal convergence rates with respect to the overall computational cost.
  Numerical experiments underpin the theoretical findings.
\end{abstract}

\section{Introduction}

\subsection{Motivation and context}

This paper concerns the numerical approximation of the solution to the quasilinear elliptic partial differential equation (PDE)
\begin{equation}\label{eq:PDEstrong}
  -\div (\mu(|\nabla \starred{u}|^2) \, \nabla \starred{u}) = f - \div \vec{f}
  \;\:\text{in }\Omega\quad\text{ subject to }\quad
  \starred{u} = 0\;\:\text{on }\partial\Omega,
\end{equation}
where $\mu\colon[0, +\infty)\to [0, +\infty)$ is a continuously differentiable scalar nonlinearity and $\Omega \subset \R^d$ is a bounded polyhedral Lipschitz domain with $d\in\N$.
More precisely, let $\XX\coloneq H_0^1(\Omega)$ with its dual space $\XX'=H^{-1}(\Omega)$ and duality pairing $\dual{\,\cdot\,}{\,\cdot\,} \coloneqq \dual{\,\cdot\,}{\,\cdot\,}_{\XX' \times \XX}$ be equipped with an equivalent norm $\enorm{ \, \cdot \,} \coloneqq a(\, \cdot \,, \, \cdot \,)^{1/2}$ that is induced by a scalar product $a(\, \cdot \,, \, \cdot \,)$.
Given $f \in L^2(\Omega)$ and $\f \in [L^2(\Omega)]^d$, we aim to approximate the weak solution $\starred{u}\in \XX$ to~\eqref{eq:PDEstrong}, which satisfies
\begin{equation}\label{eq:nonlinearPDE}
  \dual{\AA \starred{u}}{v}
  \coloneq
  \dual{\mu(|\nabla \starred{u}|^2) \nabla \starred{u}}{\nabla v}_{L^2(\Omega)}
  =
  \dual{f}{v}_{L^2(\Omega)} + \dual{\f}{\nabla v}_{L^2(\Omega)}
  \eqqcolon \dual{F}{v}
  \quad \text{for all } v \in \XX.
\end{equation}
The nonlinearity $\mu$ is assumed to fulfill, for constants $0<\widetilde{\alpha}\leq \widetilde{L}$, the growth condition
\begin{equation}\label{assumption:mu}
  \widetilde{\alpha} \, (t - s) \leq \mu(t^2) t - \mu(s^2) s \leq \widetilde{L} \, (t - s)
  \quad \text{for all } 0 \leq s \leq t.
\end{equation}
Under condition~\eqref{assumption:mu}, the (possibly) nonlinear operator $\AA \colon \XX \to \XX'$ fits into the Browder--Minty setting; cf.~\cite[Theorem~26.A]{zeidler}.
That is, $\AA$ is \emph{strongly monotone} and \emph{Lipschitz continuous}, i.e., there exist $0<\alpha\leq L$ such that
\begin{equation}\label{assumption:BM}
  \alpha \,
  \enorm{v-w}^2
  \le
  \dual{\AA v - \AA w}{v - w}
  \quad\text{and}\quad
  \dual{\AA v - \AA w}{u}
  \le L \, \enorm{v-w} \, \enorm{u}
  \quad \text{ for all } u, v, w \in \XX.
\end{equation}
We note that for $\enorm{\,\cdot\,}=\norm{\nabla (\,\cdot \,)}_{L^2(\Omega)}$, there holds $\alpha=\widetilde{\alpha}$ and $L=\widetilde{L}$, whereas $\alpha$ and $L$ depend also on norm equivalence for general $\enorm{\,\cdot\,}$; see Appendix~\ref{appendix:sharpLipschitz}.
This setting yields existence and uniqueness of the weak solution $\starred{u}\in\XX$ to~\eqref{eq:nonlinearPDE}.

This work presents an \emph{iterative Galerkin method} for approximating the solution $\starred{u}$ to~\eqref{eq:nonlinearPDE} according to the principle \emph{discretize first, linearize second}, i.e., the nonlinear PDE is first discretized and the resulting finite-dimensional nonlinear system is solved via iterative linearization.
The discretization employs a conforming adaptive finite element method (AFEM) in the spirit of~\cite{cw2017}.
The numerical treatment of strongly monotone and Lipschitz continuous problems, steered by residual-based or equilibrated-flux \textsl{a~posteriori} error estimators associated with~\eqref{eq:nonlinearPDE}, has significantly matured in the last years; see, e.g.,~\cite{gmz2011, gmz2012, ev2013, hw2020:ailfem, ghps2021, hpw2021,hpsv2021, h2023, MV23, MPS25, BPS25, harnist:hal-04033438} and the references therein.
The main novelty of this paper is a thorough convergence analysis of an algorithm that is steered by an \emph{elliptic reconstruction} estimator.
This estimator originates from the literature on parabolic problems~\cite{MN03} and is a main tool in the \textsl{a posteriori} error analysis presented in~\cite{cw2017, hw2020:convergence}.
As explained in section~\ref{subsec:reconstructionIntro} below, it yields an \textsl{a~posteriori} bound in which the linearization and discretization errors are well separated and both contributions are computable.
The reconstruction estimator arises naturally as the residual-based error estimator associated with the current linearized problem and thus may also be employed to drive adaptive mesh refinement in Uzawa-type methods used, e.g., in optimization.
Such methods adopt the reverse order \emph{linearize first, discretize second}, i.e., the nonlinear PDE is linearized at the continuous level and the resulting linearized problem is subsequently discretized.

\subsection{Iterative Galerkin method and elliptic reconstruction}
\label{subsec:reconstructionIntro}
The discretization employs conforming Lagrange finite element spaces $\XX_\coarse\subset\XX$ associated to a triangulation $\TT_H$ of $\Omega$; see section~\ref{subsection:mesh-refinement} below.
Then, the discrete formulation of~\eqref{eq:nonlinearPDE} seeks $\starred{u}_\coarse \in \XX_\coarse$ such that
\begin{equation}\label{eq:nonlinearPDE:discrete}
  \dual{\AA \starred{u}_\coarse}{v_\coarse}
  =
  \dual{F}{v_\coarse}
  \quad \text{ for all } v_\coarse \in \XX_\coarse.
\end{equation}
Existence and uniqueness of both $\starred{u} \in \XX$ and $\starred{u}_\coarse \in \XX_\coarse$ follow from the Browder--Minty theorem, cf.~\cite[Theorem~26.A]{zeidler}.
Iterative Galerkin methods solve~\eqref{eq:nonlinearPDE:discrete} by linearization via fixed-point or Newton-type iterations; see, e.g.,~\cite{hw2020:convergence, ghps2021, hpw2021} for a Zarantonello approach,~\cite{gmz2011, hpw2021, DS25} for a Ka\v{c}anov approach, or~\cite{hw2020:convergence, h2023, BBP25} for Newton's method.

In this work, we focus on a (damped) Zarantonello iteration~\cite{zarantonello1960} with damping parameter $\delta>0$.
Starting from an initial guess $u_H^0\in\XX_H$, the iteration generates a sequence of approximations $(u_\coarse^k)_{k\in\N_0}$.
For $k\geq 1$, new iterates are defined by $u_\coarse^k\coloneq u_\coarse^{k-1}+\delta \starred{z}_H^k$, where the update $\starred{z}_H^k=\starred{z}_H[u_\coarse^{k-1}] \in \XX_\coarse$ solves the linear problem
\begin{equation}\label{eq:iterativeGalerkin}
  a(\starred{z}_H^k, v_\coarse)
  =
  \dual{F-\AA u_\coarse^{k-1}}{v_\coarse}
  \quad \text{ for all }v_\coarse \in \XX_\coarse.
\end{equation}
For a sufficiently small damping parameter $0<\delta<2\alpha/L^2$ depending only on the constants in~\eqref{assumption:BM}, the iteration contracts the error in the $a(\, \cdot \,, \, \cdot \,)^{1/2}$-induced energy norm $\enorm{\,\cdot\,}$ and ensures convergence to $\starred{u}_\coarse \in \XX_\coarse$; cf.
\ section~\ref{subsection:linearization} below for details.

The key ingredient of our analysis is the following elliptic reconstruction~associated with~\eqref{eq:iterativeGalerkin}.
For $k\geq 1$, find $\starred{z}^k \in \XX$ such that
\begin{equation}\label{eq:ellipticReconstruction}
  a(\starred{z}^k, v)
  =
  \dual{F-\AA u_\coarse^{k-1}}{v}
  \quad \text{ for all }v \in \XX.
\end{equation}
The term \emph{reconstruction} reflects the fact that $\starred{z}_H^k \in \XX_\coarse$ from~\eqref{eq:iterativeGalerkin} is the Galerkin approximation of $\starred{z}^k \in \XX$ from~\eqref{eq:ellipticReconstruction}.
This is used to adaptively steer the mesh refinement.
We employ the \emph{residual}-based \textsl{a~posteriori} error estimator $\zeta_\coarse(u_\coarse^{k-1}; \starred{z}_H^k)$ corresponding to the linearized problems~\eqref{eq:iterativeGalerkin}--\eqref{eq:ellipticReconstruction}; see section~\ref{subsection:reconstruction} below.
Here, the first argument specifies the linearization point and the second argument is the discrete update whose discretization error is estimated.
Following~\cite[Proposition~2.2]{cw2017}, the update-based reconstruction estimator allows to control the error of the iterate $u_\coarse^k$ to the exact solution $\starred{u}$ of~\eqref{eq:nonlinearPDE} via
\begin{equation}\label{eq:aposteriorioEstimate}
  \enorm{\starred{u} - u_\coarse^k}
  \le
  C \, \big( \enorm{\starred{z}_H^k}
  +
  \zeta_\coarse(u_\coarse^{k-1}; \starred{z}_H^k) \big) \quad \text{ for some generic constant }
  C>0.
\end{equation}
The contribution $\enorm{\starred{z}_H^k}$ is a measure of the linearization error and $\zeta_\coarse(u_\coarse^{k-1}; \starred{z}_H^k)$ bounds the discretization error of the linearized problem.
Both terms in this \textsl{a~posteriori} error estimate are computable and the proposed algorithm aims to balance these two contributions; see section~\ref{subsection:ideaErrorControl} below for further details and a heuristic motivation.

\subsection{Main results and contributions}

The present paper provides a rigorous convergence analysis of an adaptive iterative Galerkin method that employs the reconstruction estimator $\zeta_\coarse(u_\coarse^{k-1}; \starred{z}_H^k)$, which has been missing in the literature so far.
In particular, Theorem~\ref{theorem:linearConvergence} below exploits the inherent energy structure of problem~\eqref{eq:PDEstrong} and establishes unconditional full R-linear convergence of the adaptive algorithm, i.e., contraction of a suitably defined quasi-error for arbitrary choices of adaptivity parameters.
More precisely, the quasi-error $\Zeta_\ell^k$ consists of the linearization error and the discretization error on a triangulation $\TT_\ell$ and satisfies a quasi-contraction with respect to the lexicographic order $\abs{\ell,k}$ of index pairs appearing in the algorithm, i.e., there exist $0<\qlin<1$ and $\Clin>0$ such that
\begin{equation*}
  \Zeta_{\ell}^{k}
  \le
  \Clin \, \qlin^{|\ell, k|-|\ell',k'|} \, \Zeta_{\ell'}^{k'}
  \quad \text{for all } (\ell,k), (\ell',k') \text{ with } |\ell',k'| \leq |\ell, k|.
\end{equation*}
Moreover, we prove in Theorem~\ref{theorem:optimalRates} below that, for sufficiently small adaptivity parameters, the algorithm steers the mesh refinement and the number of linearization iterations to guarantee optimal algebraic convergence rates with respect to the degrees of freedom as well as the overall computational cost.
That is, if the exact solution $\starred{u}$ is theoretically approximable with algebraic rate $r>0$ (in terms of the usual approximation classes~\cite{ckns2008,axioms}), then this rate is achieved by the algorithm with respect to the complexity, i.e.,
\begin{equation*}
  \Zeta_{\ell}^{k}
  \leq
  C \, \big(\cost(\ell,k)\big)^{-r}
  \quad \text{ for all $(\ell,k)$ and some $C>0$,}
\end{equation*}
where $\cost(\ell,k)$ denotes the (idealized) computational cost up to step $(\ell,k)$ in the algorithm.

\subsection{Organization of the paper}

The paper is organized as follows.
Section~\ref{section:iterativeLinearization} discusses the discretization of the model problem, the Zarantonello linearization and the considered error estimators, and elaborates their properties.
In section~\ref{section:adaptiveAlgorithm}, we present the algorithm of the iterative Galerkin method.
Section~\ref{section:convergenceAnalysis} states the main theorems of our convergence analysis: Unconditional full R-linear convergence in Theorem~\ref{theorem:linearConvergence} and optimal convergence rates with respect to the computational cost in Theorem~\ref{theorem:optimalRates}.
Their proofs are found in full detail in section~\ref{section:proofs}.
Section~\ref{section:experiments} concludes the paper with numerical experiments that underpin the theoretical results.

\section{Setting}\label{section:iterativeLinearization}

This section introduces the main ingredients of the adaptive algorithm: mesh-refinement, the Zarantonello iteration, and the elliptic reconstruction estimator together with its properties.

\subsection{Discretization}
\label{subsection:mesh-refinement}

Let $\TT_0$ be a given conforming initial triangulation of $\Omega$ into compact simplices.
As mesh refinement strategy, we employ newest vertex bisection (NVB); see~\cite{affkp2013} for $d=1$,~\cite{s2008} for admissible $\TT_0$ in $d\geq 2$,~\cite{kpp2013} for arbitrary $\TT_0$ in $d=2$, and~\cite{dgs2025} for arbitrary $\TT_0$ in $d\geq 2$.
Let $\T(\TT_H)$ denote the set of all conforming triangulations that can be obtained from a triangulation $\TT_H$ through finitely many steps of NVB refinement, and, for convenience, $\T\coloneq\T(\TT_0)$.
Furthermore, given a triangulation $\TT_H\in\T$ and a set of marked elements $\MM_H\subseteq\TT_H$, we denote by $\refine(\TT_H,\MM_H)$ the coarsest NVB refinement of $\TT_H$ such that at least all marked elements are refined, i.e., $\MM_H \subseteq \TT_H\setminus \refine(\TT_H,\MM_H)$.
Finally, let $p\in\N$ be a fixed polynomial degree.
To each triangulation $\TT_H \in \T$, we associate the conforming Lagrange finite element space
\begin{equation}\label{eq:FEspace}
  \XX_\coarse \coloneqq \mathcal{S}_0^p(\TT_\coarse) \coloneqq \set{v_\coarse \in \XX \given \forall T \in \TT_\coarse\colon v_\coarse|_T  \text{ is a polynomial of degree} \le p}\subset\XX.
\end{equation}
While the convergence analysis from section~\ref{section:convergenceAnalysis} below is restricted to the lowest-order case $p=1$, the algorithmic design and the numerical experiments in section~\ref{section:experiments} are carried out for polynomial degrees $p\geq 1$.
Clearly, $\XX_H \subset \XX$ is a closed subspace.
Since NVB is employed, the discrete spaces are nested, i.e., $\TT_h \in \T(\TT_H)$ implies $\XX_H \subseteq \XX_h$.

\subsection{Standard \textsl{a~posteriori} error estimator}
\label{sec:residuadlEstimator}

In the usual AFEM context, the discretization is steered by the standard residual-based \textsl{a~posteriori} error indicator $\eta_H$ associated to the nonlinear problems~\eqref{eq:nonlinearPDE} and~\eqref{eq:nonlinearPDE:discrete}; see, e.g.,~\cite{ghps2021, hpw2021, MPS25}.
Additionally, we suppose that $\vec f|_T\in [H^1(T)]^d$ for all $T\in\TT_0$ to ensure that the residual estimator below is well defined.
For a triangulation $\TT_H\in\T$ and each element $T\in\TT_H$, the local contributions then read
\begin{equation}\label{eq:residualEstimator}
  \begin{aligned}
    \eta_\coarse(T, v_H)^2
    \coloneqq
    |T|^{2/d} \, & \norm{- \div (\mu(|\nabla v_H|^2) \, \nabla v_H - \boldsymbol{f} ) - f}_{L^2(T)}^2                                                    \\
                 & + |T|^{1/d} \, \norm{\jump{[ \mu(|\nabla v_H|^2) \,\nabla v_H - \boldsymbol{f}]\cdot \boldsymbol{n}}}_{L^2(\partial T \cap \Omega)}^2
    \quad\text{for all $v_H\in\XX_H$},
  \end{aligned}
\end{equation}
where $\boldsymbol{n}$ denotes the unit normal vector on $\partial T$ and $\jump{\, \cdot \,}$ is the jump across element interfaces. The error estimator $\eta_H(v_H)$ is given by the sum of its contributions, i.e.,
\begin{equation}\label{eq:estimatorNotation}
  \eta_\coarse(\UU_\coarse, v_H)
  \coloneqq
  \biggl(\sum_{T \in \UU_\coarse} \eta_\coarse(T, v_H)^2\biggr)^{1/2}
  \quad \text{ for all } \UU_\coarse \subseteq \TT_\coarse
  \quad \text{and} \quad
  \eta_\coarse(v_H) \coloneqq \eta_\coarse(\TT_\coarse, v_H).
\end{equation}
Recall that the residual-based error estimator $\eta_H$ satisfies the following properties.

\begin{lemma}[axioms of adaptivity~{\cite[section~10.1]{axioms}}]
  \label{lemma:axioms}
  For lowest-order $p=1$, there exist constants $\Cstabeta, \Creleta, \Cdreleta,
    \Cmoneta>0$ and $0 < \qred < 1$ such that, for any
  triangulation $\TT_H \in \T$, any refinement $\TT_h
    \in\T(\TT_H)$, any subset $\UU_H \subseteq \TT_h \cap \TT_H$, and arbitrary $v_H \in \XX_H$, $v_h \in \XX_h$, the residual-based error estimator $\eta$ from~\eqref{eq:residualEstimator}--\eqref{eq:estimatorNotation} satisfies the following properties:
  \begin{enumerate}[font=\upshape, label=\textbf{\textrm{(A\arabic*)}}, ref=A\arabic*]
    \item
      \label{axiom:stability}
      \textbf{\textup{Stability:}}
      \(\mkern100mu \vert \eta_h (\UU_H; v_h) -
      \eta_H(\UU_H; v_H) \vert
      \le \Cstabeta \, \enorm{v_h - v_H}\);
    \item \label{axiom:reduction}\textbf{\textup{Reduction:}}
      \(\mkern168mu \eta_h(\TT_h \setminus \TT_H; v_H)
      \le
      \qred \, \eta_H(\TT_H \setminus \TT_h; v_H) \);
    \item
      \label{axiom:reliability}
      \textbf{\textup{Reliability:}}
      \(\mkern214mu \enorm{\starred{u} - \starred{u}_\coarse}
      \le
      \Creleta \, \eta_H(\starred{u}_\coarse)\);
    \item[\textbf{(A4)}]\refstepcounter{enumi}
      \label{axiom:discrete_reliability}
      \textbf{\textup{Discrete reliability:}}
      \(\mkern128mu
      \enorm{\starred{u}_{\fine} - \starred{u}_\coarse}
      \le
      \Cdreleta \, \eta_H(\TT_H \backslash \TT_h, \starred{u}_\coarse) \);
      \renewcommand{\theenumi}{QM}
    \item[\textbf{(QM)}]\refstepcounter{enumi}
      \label{axiom:qm}
      \textbf{\textup{Quasi-monotonicity:}}
      \(\mkern150mu \eta_h(\starred{u}_{\fine}) \le \Cmoneta \,\eta_H(\starred{u}_\coarse) \).
  \end{enumerate}
  The constant $\Cstabeta$ depends only on the uniform shape regularity of $\T$, the dimension $d$, the Lipschitz constant $L$, and the constants in the norm equivalence between $\enorm{ \, \cdot \,} = a(\, \cdot \,, \, \cdot \,)^{1/2}$ and $\norm{\nabla(\,\cdot\,)}_{L^2(\Omega)}$.
  The constants $\Creleta, \Cdreleta$ depend only on the uniform shape regularity of $\T$, the dimension $d$, the monotonicity constant $\alpha$, and the constants in the norm equivalence between $\enorm{ \, \cdot \,}$ and $\norm{\nabla(\,\cdot\,)}_{L^2(\Omega)}$.
  Furthermore, there hold $\qred = 2^{-1/(2d)}$ and $\Cmoneta \le 1 + \Cstabeta \, \Cdreleta$.
  \qed
\end{lemma}
\begin{remark}\label{rem:axioms}
  Reduction~\eqref{axiom:reduction}, reliability~\eqref{axiom:reliability}, and discrete reliability~\eqref{axiom:discrete_reliability} hold for all polynomial degrees $p\in\N$ and within the Browder--Minty setting beyond scalar nonlinearities $\mu(|\nabla u|^2) \nabla u$ satisfying the growth condition~\eqref{assumption:mu}.
  To date, stability~\eqref{axiom:stability} is known only for scalar nonlinearities satisfying~\eqref{assumption:mu} and $p=1$; cf.~\cite{gmz2012}.
\end{remark}

\subsection{Iterative Zarantonello linearization}
\label{subsection:linearization}

For $\delta>0$, the continuous Zarantonello iteration $\Phi\colon \XX \to \XX$ is defined by $\Phi(w)\coloneq w+\delta \starred{z}[w]$, where the update $\starred{z}[w]\in\XX$ is the unique solution of
\begin{equation}\label{eq:zarantonello}
  a(\starred{z}[w],v)
  =
  \dual{F- \AA w }{v}
  \quad \text{for all } v \in \XX.
\end{equation}
For a triangulation $\TT_H\in\T$, the discrete Zarantonello iteration $\Phi_H\colon \XX_H \to \XX_H$ is given analogously by $\Phi_H(w_H)\coloneq w_H+\delta \starred{z}_H[w_H]$, where the update $\starred{z}_H[w_H]\in\XX_H$ solves
\begin{equation}\label{eq:zarantonelloDiscrete}
  a(\starred{z}_H[w_H],v_H)
  =
  \dual{F- \AA w_H }{v_H}
  \quad \text{for all } v_H \in \XX_H.
\end{equation}
By construction, $\starred{z}_H[w_H]$ is the Galerkin approximation in $\XX_H$ of the continuous update $\starred{z}[w_H]$.
The existence and uniqueness of the solutions to~\eqref{eq:zarantonello}--\eqref{eq:zarantonelloDiscrete} is ensured by the Riesz theorem.

The key feature of the Zarantonello iteration (also known as the Banach--Picard iteration) is its uniform contraction.
This result originates in~\cite{zarantonello1960}; see also~\cite[Theorem~25.B]{zeidler} as a standard reference:
For $0<\delta < 2\alpha/L^2$, there exists a contraction constant $0 < \qctr=\qctr[\delta] < 1$ such that, for all $v, w \in \XX$ and $v_H, w_H \in \XX_H$, it holds that
\begin{equation}\label{eq:zarantonelloContraction}
  \enorm{\Phi(v) - \Phi(w)}
  \le
  \qctr \, \enorm{v - w}
  \quad\text{ and }\quad
  \enorm{\Phi_H(v_H) - \Phi_H(w_H)}
  \le
  \qctr \, \enorm{v_H - w_H}.
\end{equation}
More precisely, $\qctr[\delta]^2=1-\delta(2\alpha-\delta L^2)$.
For known $\alpha$ and $L$, $\qctr[\delta]$ is minimized at $\delta=\alpha/L^2$, yielding $\qctr[\delta]^2=1-\alpha^2/L^2$.
In particular, $\Phi$ and $\Phi_H$ admit unique fixed points that coincide with the solutions $\starred{u} \in \XX$ and $\starred{u}_H \in \XX_H$ of~\eqref{eq:nonlinearPDE} and~\eqref{eq:nonlinearPDE:discrete}, respectively.

\begin{remark}[alternative linearization schemes]
    \label{rem:alternativeSchmes}
    The present analysis can also be extended to alternative linearization schemes.
    Examples include the (modified) Ka\v{c}anov iteration~\cite{Kac59,HW22} and the damped Newton method~\cite{Deu04,hw2020:convergence}.
    These iterations replace the fixed scalar product $a(\, \cdot \,, \, \cdot \,)$ in~\eqref{eq:zarantonello}--\eqref{eq:zarantonelloDiscrete} by scalar products that depend on the current linearization point.
    More precisely, for any linearization point $u\in\XX$ and for all $v,w\in\XX$, these are given by
    \begin{align*}
      a_{\textup{Ka\v{c}anov}}(u;v,w)
       & \coloneq
      \dual{\mu(\abs{\nabla u}^2)\nabla v}{\nabla w}_{L^2(\Omega)}, \\
      a_{\textup{Newton}}(u;v,w)
       & \coloneq
      \dual{
        \mu(\abs{\nabla u}^2)\nabla v
        +2\mu'(\abs{\nabla u}^2)\nabla u\nabla u^{\intercal}\nabla v
      }{\nabla w}_{L^2(\Omega)}.
    \end{align*}
    The convergence analysis of these linearization iterations exploits the energy structure of the model problem~\eqref{eq:PDEstrong} and replaces the fixed-norm contraction~\eqref{eq:zarantonelloContraction} by an appropriate energy contraction; see, e.g.,~\cite{hpw2021}.
    This energy contraction and the scalar products that depend on the linearization point can be incorporated into the present analysis with corresponding modifications.
    The damping strategies from~\cite{HW22,h2023} ensure an \emph{unconditional} energy contraction in the sense that no sufficiently small fixed damping parameter needs to be prescribed in advance.
    Hence, the restriction to a fixed sufficiently small damping parameter $\delta<2\alpha/L^2$ can be avoided by employing these alternative linearization schemes.
    To maintain the focus on the novel arguments concerning the reconstruction estimator, we restrict the presentation to the Zarantonello iteration.
  \end{remark}

\subsection{Elliptic reconstruction}
\label{subsection:reconstruction}

To derive a strong formulation of~\eqref{eq:zarantonello}, we assume that the equivalent scalar product $a(\cdot, \cdot)$ is induced by some diffusion process $\vec{A}\in W^{1,\infty}(\Omega; \R^{d\times d}_{\rm sym})$, i.e., $a(v, w)=\dual{\vec{A}\nabla v}{\nabla w}_{L^2(\Omega)}$.
For a linearization point $w_H\in\XX_H$, consider the discrete Zarantonello update $\starred{z}_H[w_H]$ solving~\eqref{eq:zarantonelloDiscrete}.
Its elliptic reconstruction $\starred{z}[w_H]$ from~\eqref{eq:zarantonello} solves in a weak sense
\begin{equation}\label{eq:linearizedProblem:auxiliary}
  -\div (\vec{A}\nabla\starred{z}[w_H])
  =
  f + \div\big(\mu(|\nabla w_H|^2) \, \nabla w_H -\boldsymbol{f}\big).
\end{equation}
We steer the adaptive mesh refinement by the residual-based error estimator $\zeta_H$ associated to the linear problem~\eqref{eq:linearizedProblem:auxiliary}.
Under the additional regularity assumption on $\vec{f}$ from
section~\ref{sec:residuadlEstimator}, its local contributions read, for $w_H\in\XX_H$ and $T\in\TT_H$,
\begin{equation}\label{eq:ellipticReconstructionEstimator}
  \begin{aligned}
     & \zeta_H(w_H; T, z_H)^2
    \coloneqq
    |T|^{2/d} \, \norm{-\div (\vec{A}\nabla z_H) - f - \div\big(\mu(|\nabla w_H|^2) \, \nabla w_H -\boldsymbol{f}\big)
    }_{L^2(T)}^2
    \\
     & \qquad\qquad\quad +
    |T|^{1/d} \,
    \norm{\jump{
        [
            \vec{A}\nabla z_H
            +
            \mu(|\nabla w_H|^2) \,\nabla w_H - \boldsymbol{f}
          ]
        \cdot \boldsymbol{n}}}_{L^2(\partial T \cap \Omega)}^2
    \qquad\text{for all $z_H\in\XX_H$}.
  \end{aligned}
\end{equation}
The notation $\zeta_H(w_H; \UU_H, z_H)$ and $\zeta_H(w_H; z_H)\coloneqq \zeta_H(w_H; \TT_H, z_H)$ is used analogously to~\eqref{eq:estimatorNotation}.

\begin{proposition}[estimator equivalence]\label{prop:estimatorEquivalence}
  For lowest-order $p=1$, there exists a constant $\Ceqv>0$ such that, for all $\delta>0$ and all $\UU_H\subseteq\TT_H$, the \textsl{a~posteriori} error estimators $\eta_H$ from~\eqref{eq:residualEstimator} and $\zeta_H$ from~\eqref{eq:ellipticReconstructionEstimator} satisfy
  \begin{equation}\label{eq:estimatorEquivalence}
    \abs{\zeta_H(w_H; \UU_H, z_H)-\eta_H(\UU_H, w_H+\delta z_H)}\leq
    (1+\delta)\,\Ceqv \, \enorm{z_H}\quad\text{for all $z_H, w_H \in \XX_H$}.
  \end{equation}
  The constant $\Ceqv$ depends only on the shape regularity of $\TT_0$, the dimension $d$, $\Cstab$, and the constants in the norm equivalence between $\enorm{ \, \cdot \,} = a(\, \cdot \,, \, \cdot \,)^{1/2}$ and $\norm{\nabla(\,\cdot\,)}_{L^2(\Omega)}$.
\end{proposition}
\begin{proof}
  Set $v_H\coloneq w_H+\delta z_H$.
  Since $p=1$, we have
  \begin{equation*}
    \div\big(\mu(|\nabla v_H|^2) \, \nabla v_H\big)=0=\div\big(\mu(|\nabla w_H|^2) \, \nabla w_H\big)
    \quad\text{ on each element $T\in\TT_H$}.
  \end{equation*}
  Therefore, the inverse triangle inequality yields
  \begin{equation}\label{lem:estimatoreqv:eq1}\begin{aligned}
       & \abs{\zeta_H(w_H; \UU_H, z_H)-\eta_H(\UU_H, w_H+\delta z_H)}^2=\abs{\zeta_H(w_H; \UU_H, z_H)-\eta_H(\UU_H, v_H)}^2                                   \\
       & \qquad\leq\sum_{T\in\UU_H}\Bigl( |T|^{2/d} \,\norm{\div (\vec{A}\nabla z_H)}_{L^2(T)}^2                                                              \\
       & \qquad\qquad\qquad\quad+|T|^{1/d} \,\norm{\jump{[\vec{A}\nabla z_H + \mu(|\nabla w_H|^2) \,\nabla w_H-\mu(|\nabla v_H
              |^2) \,\nabla v_H] \cdot \boldsymbol{n}}}_{L^2(\partial T \cap \Omega)}^2\Bigr)
      \\
       & \qquad \leq \sum_{T\in\UU_H}
      \Bigl(
      |T|^{2/d} \,\norm{\div (\vec{A}\nabla z_H)}_{L^2(T)}^2+|T|^{1/d} \,\norm{\jump{\vec{A}\nabla z_H \cdot \boldsymbol{n}}}_{L^2(\partial T \cap \Omega)}^2 \\
       & \qquad\qquad\qquad\quad+
      |T|^{1/d} \,\norm{\jump{ [\mu(|\nabla w_H|^2) \,\nabla w_H-\mu(|\nabla v_H|^2) \,\nabla v_H ]\cdot \boldsymbol{n}}}_{L^2(\partial T \cap \Omega)}^2\Bigr).
    \end{aligned}\end{equation}
  The proof of~\eqref{axiom:stability} in~\cite[Lemma~3.3]{gmz2012} shows that
  \begin{equation}\label{lem:estimatoreqv:eq2}\begin{aligned}
      \sum_{T\in\UU_H}|T|^{1/d} \,\norm{\jump{ [\mu(|\nabla w_H|^2) \,\nabla w_H-\mu(|\nabla v_H|^2) \,\nabla v_H ]\cdot \boldsymbol{n}}}_{L^2(\partial T \cap \Omega)}^2 & \leq \Cstabeta\, \enorm{w_H - v_H}^2.
    \end{aligned}\end{equation}
  The norm equivalence of $\enorm{ \, \cdot \,}$ and $\norm{\nabla(\,\cdot\,)}_{L^2(\Omega)}$, standard trace and inverse inequalities, the shape regularity of all triangulations, and $\vec{A}\in W^{1,\infty}(\Omega; \R^{d\times d}_{\rm sym})$ yield
  \begin{equation}\label{lem:estimatoreqv:eq3}
    \sum_{T\in\UU_H}\Bigl(
    |T|^{2/d} \,\norm{\div (\vec{A}\nabla z_H)}_{L^2(T)}^2+|T|^{1/d} \,\norm{\jump{ \vec{A}\nabla z_H \cdot \boldsymbol{n}}}_{L^2(\partial T \cap \Omega)}^2\Bigr)
    \lesssim
    \enorm{z_H}^{2}.
  \end{equation}
  The combination of~\eqref{lem:estimatoreqv:eq1}--\eqref{lem:estimatoreqv:eq3} concludes the proof of~\eqref{eq:estimatorEquivalence} via
  \begin{equation*}
    \abs{\zeta_H(w_H; \UU_H, z_H)-\eta_H(\UU_H, w_H+\delta z_H)}\lesssim\enorm{z_H}+\enorm{w_H - v_H}=(1+\delta)\, \enorm{z_H}.\qedhere
  \end{equation*}
\end{proof}

\begin{remark}\label{rem:equivalenceToCW}
  The update-based reconstruction estimator $\zeta_H(w_H;T,z_H)$ from~\eqref{eq:ellipticReconstructionEstimator} coincides, up to a scaling factor, with the estimator considered in~\cite{cw2017}.
  More precisely, the estimator $\zeta_H^{\rm CW}(w_H;T,v_H)$ from~\cite{cw2017} is based on the reconstruction $\Phi(w_H)\approx\Phi_H(w_H)$ of the Galerkin approximation $\Phi_H(w_H)$ and, in the present setting, would read, for all $T\in\TT_H$ and $v_H, w_H\in\XX_H$,
  \begin{equation*}
    \begin{aligned}
      \zeta_H^{\rm CW}(w_H; T, v_H)^2
       & \coloneqq
      |T|^{2/d} \, \norm{-\div (\vec{A}\nabla( v_H- w_H))
        - \delta \,
        [f + \div(\mu(|\nabla w_H|^2) \, \nabla w_H - \boldsymbol{f}) ]
      }_{L^2(T)}^2
      \\
       & \qquad +
      |T|^{1/d} \,
      \norm{\jump{
          [
              \vec{A}\nabla(v_H - w_H)
              + \delta
              (\mu(|\nabla w_H|^2) \,\nabla w_H - \boldsymbol{f}  )
            ]
          \cdot \boldsymbol{n}}}_{L^2(\partial T \cap \Omega)}^2.
    \end{aligned}
  \end{equation*}
  It is easy to check that, for all $T\in\TT_H$ and all $z_H, w_H\in\XX_H$,
  \begin{equation*}
    \zeta_H^{\rm CW}(w_H; T, w_H+\delta z_H)=\delta \,\zeta_H(w_H; T, z_H).
  \end{equation*}
  Thus, the two estimators agree element-wise up to the factor $\delta$ and hence yield the same set of marked elements under Dörfler marking.
  The convergence analysis developed in the subsequent sections therefore immediately carries over to the estimator $\zeta_H^{\rm CW}$.
  In addition, the update-based reconstruction estimator has the advantage of decoupling the estimator from the damping parameter~$\delta$.
\end{remark}

\begin{remark}
  For $w_H\in\XX_H$, the reconstruction estimator $\zeta_H$ from~\eqref{eq:ellipticReconstructionEstimator} satisfies the \emph{axioms of adaptivity} from Lemma~\ref{lemma:axioms} with constants that depend on the PDE-operator in the usual sense, but that are independent of the linearization point~$w_H$.
  For the continuous update $\starred{z}[w_H]\in\XX$ and its Galerkin approximation $\starred{z}_H[w_H]\in\XX_\coarse$, reliability reads
  \begin{equation}\label{eq:reliabilityZeta}
    \enorm{\starred{z}[w_H] - \starred{z}_H[w_H]}\leq\Crelzeta \, \zeta_H(w_H; \starred{z}_H[w_H])\quad\text{ for some $\Crelzeta>0$.}
  \end{equation}
  Moreover, an analogous reasoning as in Proposition~\ref{prop:estimatorEquivalence} shows that the reconstruction estimator $\zeta_H$ is stable with respect to both arguments, i.e., for all $\UU_H\subseteq\TT_H$ and $z_H, \widetilde{z}_H, w_H, \widetilde{w}_H \in \XX_H$,
  \begin{equation*}
    \begin{aligned}
      |\zeta_H(w_H; \UU_H, z_H) - \zeta_H(\widetilde{w}_H; \UU_H, \widetilde{z}_H)|
       & \le
      \Cstabzeta \, \big( \enorm{w_H - \widetilde{w}_H} + \enorm{z_H - \widetilde{z}_H} \big)\;\text{ for some $\Cstabzeta>0$.}
    \end{aligned}
  \end{equation*}
  We note, however, that the convergence analysis relies only on estimator equivalence~\eqref{eq:estimatorEquivalence} and reliability~\eqref{eq:reliabilityZeta} is used only to motivate the algorithmic design.
\end{remark}

\section{Adaptive Algorithm}\label{section:adaptiveAlgorithm}

The proposed AFEM algorithm follows the \emph{discretize first, linearize second} principle, where adaptive mesh refinement is driven by the reconstruction estimator~\eqref{eq:ellipticReconstructionEstimator}.
The resulting discrete problems~\eqref{eq:nonlinearPDE:discrete} are linearized using the discrete Zarantonello iteration.
A key aspect in this approach is to balance linearization and discretization errors via a stopping criterion for the linearization loop.

\subsection{Idea of error control}\label{subsection:ideaErrorControl}

Given $\TT_H$, let $u_H^{k-1} \in \XX_H$ be a given previous iterate.
Recall that the continuous Zarantonello iterate is given by $\Phi(u_H^{k-1})=u_H^{k-1}+\delta \,\starred{z}[u_H^{k-1}] \in \XX$, where the exact update $\starred{z}[u_H^{k-1}]$ solves~\eqref{eq:zarantonello}.
Since $\Phi(\starred{u})=\starred{u}$, the contraction~\eqref{eq:zarantonelloContraction} of $\Phi$ and the triangle inequality imply
\begin{equation*}
  \enorm{\starred{u} - \Phi(u_H^{k-1})}=\enorm{\Phi(\starred{u}) - \Phi(u_H^{k-1})}
  \eqreff{eq:zarantonelloContraction}{\leq} \qctr\,\enorm{\starred{u} - u_H^{k-1}}
  \leq \qctr\,\enorm{\starred{u} - \Phi(u_H^{k-1})} + \qctr\,\enorm{\Phi(u_H^{k-1}) - u_H^{k-1}}.
\end{equation*}
Rearranging the terms, this yields the following standard \textsl{a~posteriori} estimate
\begin{equation}\label{eq:APEstim}
  \enorm{\starred{u} - \Phi(u_H^{k-1})}
  \leq \frac{\qctr}{1-\qctr}\,\enorm{\Phi(u_H^{k-1}) - u_H^{k-1}}.
\end{equation}
Let $\starred{z}_H^k \coloneq \starred{z}_H[u_H^{k-1}] \in \XX_H$ denote the discrete update defined by~\eqref{eq:zarantonelloDiscrete}.
The total error of the discrete Zarantonello iterate $u_H^k \coloneq u_H^{k-1} + \delta \, \starred{z}_H^k \in \XX_H$ with respect to the exact solution $\starred{u}$ of~\eqref{eq:nonlinearPDE} can then be controlled as follows.
The triangle inequality, the estimate~\eqref{eq:APEstim}, the definition of the discrete update $\starred{z}_H^k$, and reliability~\eqref{eq:reliabilityZeta} of $\zeta_H$ show that
\begin{equation}\label{eq:errorcontrol}\begin{aligned}
    \enorm{\starred{u} - u_H^k}
     & \leqpad\enorm{\starred{u}-\Phi(u_H^{k-1})} + \enorm{\Phi(u_H^{k-1}) - u_H^k}                                                                                         \\
     & \eqreff*{eq:APEstim}{\leqpad} \frac{\qctr}{1-\qctr}\,\enorm{\Phi(u_H^{k-1}) - u_H^{k-1}}+ \enorm{\Phi(u_H^{k-1}) - u_H^k}                            \\
     & \eqpad \frac{\qctr\,\delta}{1-\qctr}\,\enorm{\starred{z}[u_H^{k-1}] }+ \delta\,\enorm{\starred{z}[u_H^{k-1}]-\starred{z}_H^k}                                        \\
     & \leqpad \frac{\qctr\,\delta}{1-\qctr}\,\enorm{\starred{z}_H^k}+\Big( \delta + \frac{\qctr\,\delta}{1-\qctr} \Big)\,\enorm{\starred{z}[u_H^{k-1}]-\starred{z}_H^k}    \\
     & \eqreff*{eq:reliabilityZeta}{\leqpad} \frac{\qctr\,\delta}{1-\qctr}\,\enorm{\starred{z}_H^k}+\frac{\delta}{1-\qctr}\,\Crelzeta\,\zeta_H(u_H^{k-1}; \starred{z}_H^k).
  \end{aligned}\end{equation}
Thus, to balance both contributions of the total error, we employ, for a balancing parameter $\lctr>0$, the stopping criterion
\begin{equation*}
  \enorm{\starred{z}_H^k}
  \le
  \lctr \, \zeta_H(u_H^{k-1}; \starred{z}_H^k)
\end{equation*}
to adaptively steer the number of linearization steps in the inner loop of the following algorithm.

\subsection{Algorithmic formulation}

The adaptive algorithm is presented below.

\begin{algorithm}\label{algorithm}
  \textbf{Input:} initial triangulation $\TT_0$, initial guess $u_0^0 \in \XX_0$, adaptivity parameters $0 < \theta \le 1$, $\lctr>0$, and $\Cmark\geq 1$, damping parameter $\delta>0$.\\
  {\bf for all }\;$\ell = 0, 1, 2, \dots$ { \;\bf do}
  \begin{enumerate}
    \item[{\rm(i)}] {\bf for all }\;$k = 1, 2, 3, \dots$ { \;\bf do}
      \begin{enumerate}
        \item[\rm{(a)}] Compute the discrete update $\starred{z}_\ell^k\coloneq\starred{z}_\ell[u_\ell^{k-1}]\in\XX_\ell$ by solving
          \begin{equation}\label{algorithm:solve}
            a(\starred{z}_\ell^k, v_\ell) = \dual{F- \AA u_\ell^{k-1} }{v_\ell} \quad \text{for all } v_\ell \in \XX_\ell.
          \end{equation}
        \item[\rm{(b)}] Update the Zarantonello iterate via $u_\ell^k\coloneqq \Phi_\ell(u_\ell^{k-1})=u_\ell^{k-1} + \delta \, \starred{z}_\ell^k\in\XX_\ell$.
        \item[{\rm(c)}] Compute the refinement indicator $\zeta_\ell(u_{\ell}^{k-1}; T, \starred{z}_\ell^k)$ for all $T \in \TT_\ell$.
        \item[{\rm(d)}] Terminate the $k$-loop if the following stopping criterion is satisfied:
          \begin{equation}\label{algorithm:stoppingCriterion}
            \enorm{\starred{z}_\ell^k}
            \le
            \lctr \, \zeta_\ell(u_{\ell}^{k-1}; \starred{z}_\ell^k).
          \end{equation}
      \end{enumerate}
    \item[] {\bf end for}
    \item[{\rm(ii)}] Define $\kmax[\ell] \coloneqq k$ and determine a set $\MM_\ell\subseteq\TT_\ell$ with up to the factor $\Cmark$ minimal cardinality that satisfies the Dörfler marking criterion
      \begin{equation}\label{algorithm:doerfler}
        \theta\,\zeta_\ell(u_{\ell}^{\kmax[\ell]-1}; \starred{z}_\ell^{\kmax[\ell]})^2\leq \zeta_\ell(u_{\ell}^{\kmax[\ell]-1}; \MM_\ell, \starred{z}_\ell^{\kmax[\ell]})^2.
      \end{equation}
    \item[{\rm(iii)}] Generate $\TT_{\ell+1} \coloneqq \refine(\TT_\ell, \MM_\ell)$ by newest-vertex bisection.
    \item[{\rm(iv)}] Define $u_{\ell+1}^0 \coloneqq u_\ell^{\kmax[\ell]}$ \quad \textbf{\textup{(nested iteration)}}.
  \end{enumerate}
  {\bf end for}
  \\
  \textbf{Output:}
  Sequence of successively refined meshes $\TT_\ell$ with corresponding updates $\starred{z}_\ell^k\in\XX_\ell$ and approximations $u_\ell^k \in \XX_\ell$ and error estimators $\zeta_\ell(u_{\ell}^{k-1}; \starred{z}_\ell^k)$.
\end{algorithm}

For the convergence analysis of the algorithm, we define the countably infinite index set
\begin{equation*}
  \QQ \coloneqq \set{(\ell,k) \in \N_0\times\N \given \starred{z}_\ell^k \text{ is defined in Algorithm~\ref{algorithm}}},
\end{equation*}
which is naturally equipped with the order $(\ell',k')\leq(\ell,k)$ if and only if $\starred{z}_{\ell'}^{k'}$ is defined not later than $\starred{z}_\ell^k$ in the sequential Algorithm~\ref{algorithm}. For all $(\ell,k)\in\QQ$, the total step counter is given by
\begin{equation*}
  |\ell,k|
  \coloneqq
  \#\set{(\ell',k') \in \QQ \given (\ell', k') < (\ell, k)} \in \N_0.
\end{equation*}
In accordance with the notation of Algorithm~\ref{algorithm}, define
\begin{equation*}
  \lmax
  \coloneqq
  \sup \set{\ell \in \N_0 \given (\ell,1) \in \QQ}
  \quad\text{ and }\quad
  \kmax[\ell]
  \coloneqq
  \sup \set{k \in \N \given (\ell, k) \in \QQ}
  \quad \text{for all } 0 \le \ell \le \lmax.
\end{equation*}
Note that either $\lmax = \infty$ (and hence $\kmax[\ell] < \infty$ for all $\ell \in \N_0$), or $\lmax < \infty$ with $\kmax[\lmax] = \infty$.
We abbreviate $\kmax=\kmax[\ell]$ whenever the index $\ell$ is clear from the context, e.g., $u_\ell^\kmax$ instead of $u_\ell^{\kmax[\ell]}$.

\section{Convergence analysis}\label{section:convergenceAnalysis}

This section presents the convergence analysis of Algorithm~\ref{algorithm}, which relies on the energy structure of the model problem.
Regarding the error estimators, we only exploit that the standard estimator $\eta$ from~\eqref{eq:residualEstimator} satisfies~\eqref{axiom:stability}--\eqref{axiom:discrete_reliability} and that $\eta$ and the reconstruction estimator $\zeta$ from~\eqref{eq:ellipticReconstructionEstimator} are equivalent in the sense of~\eqref{eq:estimatorEquivalence}.
We refer to Remark~\ref{rem:axioms} for a detailed discussion of these estimator assumptions.

\subsection{Lucky breakdown}

In the case that $\ellu<+\infty$, the linearization loop fails to terminate on the triangulation $\TT_\ellu$.
The following lemma asserts that this occurs only if the exact solution $\starred{u}$ of~\eqref{eq:nonlinearPDE} is discrete.
\begin{lemma}
  If $\underline\ell < \infty$ (i.e., $\underline k[\underline\ell] = \infty$), there holds $\starred{u} = \starred{u}_{\underline\ell}$, $\eta_{\underline\ell}(\starred{u}_{\underline\ell}) = 0$, and $u_{\underline\ell}^{k}\to \starred{u}$ as $k \to \infty$.
\end{lemma}
\begin{proof}
  The discrete Zarantonello iteration in $\XX_{\underline\ell}$ converges to the unique solution $\starred{u}_{\underline\ell} \in \XX_{\underline\ell}$ of~\eqref{eq:nonlinearPDE:discrete}, hence $\enorm{\starred{u}_{\underline\ell} - u_{\underline\ell}^k} \to 0$ as $k \to \infty$ and, subsequently, also $\enorm{\starred{z}_{\ellu}^k}=\delta^{-1}\enorm{u_{\underline\ell}^k - u_{\underline\ell}^{k-1}}\to 0$.
  The stopping criterion in Algorithm~\ref{algorithm} is never satisfied for any $k$ and thus, there holds
  \begin{equation*}
    \zeta_{\underline\ell}(u_{\underline\ell}^{k-1}; \starred{z}_{\underline\ell}^k)
    <
    \lctr^{-1} \, \enorm{\starred{z}_{\ellu}^k}
    \xrightarrow[]{k \to \infty} 0.
  \end{equation*}
  Stability~\eqref{axiom:stability} and reliability~\eqref{axiom:reliability} of $\eta_\ellu$ and estimator equivalence~\eqref{eq:estimatorEquivalence} conclude the proof via
  \begin{equation*}
    \enorm{\starred{u} - \starred{u}_{\underline\ell}}
    \eqreff{axiom:reliability}{\lesssim}\eta_{\underline\ell}(\starred{u}_{\underline\ell})
    \eqreff{axiom:stability}{\lesssim}\eta_{\underline\ell}(u_{\underline\ell}^k) + \enorm{\starred{u}_{\underline\ell} - u_{\underline\ell}^k}
    \eqreff{eq:estimatorEquivalence}{\lesssim}\zeta_{\underline\ell}(u_{\underline\ell}^{k-1}; \starred{z}_{\underline\ell}^k) + \enorm{\starred{z}_{\underline\ell}^k} + \enorm{\starred{u}_{\underline\ell} - u_{\underline\ell}^k} \xrightarrow[]{k \to \infty} 0.\qedhere
  \end{equation*}
\end{proof}

\subsection{Unconditional full R-linear convergence}

This section presents our first main result.
It shows that Algorithm~\ref{algorithm} guarantees full R-linear convergence \emph{unconditionally} with respect to the algorithmic parameters.
This result is particularly relevant in practice, since convergence is ensured independently of the user-chosen adaptivity parameters.
We stress that the smallness condition on $\delta$ required for the Zarantonello contraction~\eqref{eq:zarantonelloContraction} does not compromise this unconditionality, since alternative linearization schemes can be applied without this restriction; cf.~Remark~\ref{rem:alternativeSchmes}.
The proof exploits the underlying energy structure of the PDE~\eqref{eq:PDEstrong}.
More precisely, the associated energy functional $\EE\colon\XX\to\R$ is given by
\begin{equation}\label{def:energy}
  \EE(v)
    \coloneq \frac{1}{2}\int_\Omega\int_0^{\abs{\nabla v(x)}^2}\mu(t) \d{t}\d{x}-\dual{F}{v}
    \quad\text{ for all }v\in\XX.
\end{equation}
The functional $\EE$ is Gâteaux differentiable with $\dual{\EE'(v)}{w}=\dual{\AA v-F}{w}$ for all $v,w\in\XX$.
It is well-known (see, e.g.,~\cite[Lemma~5.1]{GHPS18} or~\cite[Lemma~2]{hw2020:ailfem}) that strong monotonicity and Lipschitz continuity of $\AA$ in~\eqref{assumption:BM} imply the following energy equivalence
\begin{equation}\label{eq:energyEqv}
  0\leq\frac{\alpha}{2}\,\enorm{\starred{u}_H-v_H}^2
    \leq \EE(v_H)-\EE(\starred{u}_H)
    \leq\frac{L}{2}\,\enorm{\starred{u}_H-v_H}^2
    \quad
    \text{ for all $\TT_H\in\T$ and all $v_H\in\XX_H$},
\end{equation}
which also holds for $\starred{u}\in\XX$ and $\XX$ in place of $\starred{u}_H\in\XX_H$ and $\XX_H$.
In particular, this implies that the exact solution $\starred{u}\in\XX$ of~\eqref{eq:nonlinearPDE} and the discrete solution $\starred{u}_H\in\XX_H$ of~\eqref{eq:nonlinearPDE:discrete} are the unique minimizers of $\EE$ over $\XX$ and $\XX_H$, respectively.
This energy structure yields the quasi-orthogonality estimate used in Step~2 of the proof, following the approach of~\cite{gmz2012,GHPS18}.
The complete proof is given in section~\ref{subsec:proofUncond}.

\begin{theorem}[unconditional full R-linear convergence]\label{theorem:linearConvergence}
  For any $0 < \theta \le 1$ and any $\lctr>0$, Algorithm~\ref{algorithm} steered by the elliptic reconstruction error estimator $\zeta_\ell$ guarantees full R-linear convergence of the quasi-error
  \begin{equation}\label{eq:quasiError}
    \Zeta_\ell^k
    \coloneqq
    \enorm{\starred{u}_\ell - u_{\ell}^{k-1}} + \zeta_\ell(u_{\ell}^{k-1}; \starred{z}_\ell^k) \quad \text{for all } (\ell,k) \in \QQ,
  \end{equation}
  i.e., there exist constants $0 < \qlin < 1$ and $\Clin>0$ such that
  \begin{equation}\label{eq:linearConvergence}
    \Zeta_{\ell}^{k}
    \le
    \Clin \, \qlin^{|\ell, k|  -|\ell',k'|  } \, \Zeta_{\ell'}^{k'}
    \quad \text{ for all } (\ell',k'), (\ell,k) \in \QQ \text{ with }  |\ell',k'| \leq |\ell, k|.
  \end{equation}
  The constants $\qlin$ and $\Clin$ depend only on $\lambda, \qctr, \theta, \delta, \Ceqv,\alpha,L$, and the constants in~\eqref{axiom:stability}--\eqref{axiom:discrete_reliability},~\eqref{axiom:qm}.
\end{theorem}

\begin{remark}\label{rem:quasiErrors}
  Since $\starred{u}_\ell\in\XX_\ell$ is a fixed point of the discrete Zarantonello iteration, i.e., $\Phi_\ell(\starred{u}_\ell)=\starred{u}_\ell$, and $u_\ell^k=\Phi_\ell(u_\ell^{k-1})$, the triangle inequality and the contraction~\eqref{eq:zarantonelloContraction} of $\Phi_\ell$ yield
  \begin{equation}\label{rem:convergence:eq1}
    (1-\qctr)\,\enorm{\starred{u}_\ell-u_\ell^{k-1}}\eqreff{eq:zarantonelloContraction}{\leq}\enorm{u_\ell^k-u_\ell^{k-1}}=\delta\,\enorm{\starred{z}_\ell^k}\eqreff{eq:zarantonelloContraction}{\leq} (1+\qctr)\,\enorm{\starred{u}_\ell-u_\ell^{k-1}}
    \quad\text{ for all $(\ell,k)\in\QQ$.}
  \end{equation}
  Consequently, the quasi-error $\Zeta_\ell^k$ from~\eqref{eq:quasiError} is equivalent $\Zeta_\ell^k\simeq\widetilde{\Zeta}_\ell^k$ to the fully computable quantity
  \begin{equation}\label{eq:quasiErrorComputable}
    \widetilde{\Zeta}_\ell^k\coloneq\enorm{\starred{z}_\ell^k}+ \zeta_\ell(u_{\ell}^{k-1}; \starred{z}_\ell^k)\quad\text{ for all $(\ell,k)\in\QQ$}.
  \end{equation}
  In particular, by~\eqref{eq:errorcontrol}, Theorem~\ref{theorem:linearConvergence} guarantees unconditional convergence of Algorithm~\ref{algorithm} via
  \begin{equation*}
    \enorm{\starred{u}-u_\ell^k}
    \eqreff{eq:errorcontrol}{\lesssim}\enorm{\starred{z}_\ell^k}+\zeta_\ell(u_\ell^{k-1}; \starred{z}_\ell^k)
    =\widetilde{\Zeta}_\ell^k
    \simeq\Zeta_\ell^k
    \eqreff{eq:linearConvergence}{\leq}\Clin \, \qlin^{|\ell, k|  }\Zeta_0^1\:\xrightarrow{\abs{\ell,k}  \to\infty}\:0.
  \end{equation*}
\end{remark}

\begin{remark}\label{rem:condConv}
  For sufficiently small $\lctr$, Algorithm~\ref{algorithm} implicitly induces a Dörfler criterion for $\eta_\ell$.
  More precisely, if
  $0<\lctr  < (1+\delta)^{-1}\,\Ceqv^{-1}\,\theta^{1/2}\eqcolon\lctr_0$,
  then $\eta_\ell(u_\ell^{\underline k})$ from~\eqref{eq:residualEstimator} fulfills
  \begin{equation}\label{lemma:doerflerSufficient:eq1}
    \widetilde{\theta}\,\eta_\ell(u_\ell^{\underline k})^2\leq \eta_\ell(\MM_\ell, u_\ell^{\underline k})^2\quad\text{with}\quad 0 < \widetilde{\theta} \coloneq \frac{(\theta^{1/2} - \lctr (1+\delta)\Ceqv)^2}{(1+\lctr (1+\delta)\Ceqv)^2} < 1.
  \end{equation}
  A proof of~\eqref{lemma:doerflerSufficient:eq1} is found in Appendix~\ref{subsec:DoerflerSufficient}.
  Hence, for $\lctr<\lctr_0$, Algorithm~\ref{algorithm} steered by $\zeta_\ell$ falls within the scope of the established R-linear convergence theory for adaptive algorithms driven by~$\eta_\ell$; see, e.g.,~\cite{ghps2021,bfmps2025}.
  In particular,~\cite[Theorem~2 and section~6]{bfmps2025} imply \emph{conditional} full R-linear convergence of the quasi-error
  \begin{equation}\label{eq:defEtaQuasiErrir}
    \Eta_\ell^k
    \coloneqq
    \enorm{\starred{u}_\ell - u_{\ell}^{k}}
    + \eta_\ell(u_\ell^k)\quad\text{ for all $(\ell,k)\in\QQ$},
  \end{equation}
  i.e., for any $0 < \theta \le 1$ and sufficiently small $0<\lctr  < \lctr_0$, there exist $0 < \qlin < 1$ and $\Clin>0$, depending only on $\lambda, \theta, \qctr, \alpha, L$, and the constants in~\eqref{axiom:stability}--\eqref{axiom:discrete_reliability}, and~\eqref{axiom:qm}, such that
  \begin{equation*}
    \Eta_{\ell}^{k}
    \le
    \Clin \, \qlin^{|\ell, k|-|\ell',k'|} \, \Eta_{\ell'}^{k'} \quad \text{ for all } (\ell',k'), (\ell,k) \in \QQ \text{ with }  |\ell',k'| \leq |\ell, k|.
  \end{equation*}
  This implication, however, is guaranteed only under the smallness condition $\lctr<\lctr_0$.
  Moreover, estimates~\eqref{rem:convergence:eq1} together with estimator equivalence~\eqref{eq:estimatorEquivalence} and stability~\eqref{axiom:stability} of $\eta_\ell$ show that
  \begin{equation*}\begin{aligned}
      \zeta_\ell(u_{\ell}^{k-1}; \starred{z}_\ell^k)\eqreff{eq:estimatorEquivalence}{\lesssim}\eta_\ell(u_\ell^k) + \enorm{\starred{z}_\ell^k}\eqreff{axiom:stability}{\lesssim}\eta_\ell(u_\ell^{k-1})+ \enorm{u_\ell^k-u_\ell^{k-1}} + \enorm{\starred{z}_\ell^k}\eqreff{rem:convergence:eq1}{\lesssim}\Eta_\ell^{k-1}, \\
      \eta_\ell(u_\ell^{k-1})\eqreff{axiom:stability}{\lesssim}\eta_\ell(u_\ell^k) + \enorm{u_\ell^k-u_\ell^{k-1}}
      \eqreff{eq:estimatorEquivalence}{\lesssim}\zeta_\ell(u_{\ell}^{k-1}; \starred{z}_\ell^k) + \enorm{\starred{z}_\ell^k}+ \enorm{u_\ell^k-u_\ell^{k-1}}\eqreff{rem:convergence:eq1}{\lesssim}\Zeta_\ell^k,
    \end{aligned}\end{equation*}
  where $\Eta_\ell^{k}$ from~\eqref{eq:defEtaQuasiErrir} is extended by $\Eta_\ell^0\coloneq\enorm{\starred{u}_\ell - u_{\ell}^{0}}+\eta_\ell(u_\ell^0)$. In particular, this implies that
  \begin{equation}\label{rem:condConv:eq1}
    \Zeta_\ell^k\simeq\Eta_\ell^{k-1}.
  \end{equation}
  Hence, the above $R$-linear convergence result for $\Eta_\ell^k$ immediately transfers to $\Zeta_\ell^k$ whenever $\lctr<\lctr_0$.
  Theorem~\ref{theorem:linearConvergence} significantly strengthens this by establishing full R-linear convergence of $\Zeta_\ell^k$ for arbitrary $\lctr>0$ without relying on an (implicit) Dörfler criterion for $\eta_\ell$.
\end{remark}

\subsection{Optimal convergence rates and quasi-optimal complexity}

We briefly address the computational complexity of a practical realization of Algorithm~\ref{algorithm}.
For ease of presentation, the analysis assumes that the linear system~\eqref{algorithm:solve} is solved exactly, although an exact solution cannot, in general, be computed at linear cost.
The frameworks in~\cite{bfmps2025,MPS25}, however, provide suitable stopping criteria that allow a uniformly contractive algebraic solver to be incorporated into the analysis to preserve unconditional full R-linear convergence and rigorously establish quasi-optimal complexity.
Suitable solvers with linear cost per iteration include a conjugate gradient method with an optimal multilevel preconditioner and an optimal multigrid method; for further details, see, e.g.,~\cite{CNX12,WZ17,IMPS24,HMP26} and the references therein.
In particular, the parameter-free stopping criterion from~\cite{MPS25} guarantees the required energy contraction of the inexact linearization and a uniformly bounded number of algebraic iterations.
Consequently, it is reasonable to assume that a sufficiently accurate approximation of the discrete update $\starred{z}_\ell^k$ in~\eqref{algorithm:solve} can be computed in $\OO(\#\TT_\ell)$ operations.
These well-established extensions can be incorporated into the present analysis.
However, to focus on the novel arguments concerning the reconstruction estimator, we opt to omit the additional algebraic loop and the corresponding technical notation.
Up to quadrature, the evaluation of the refinement indicators likewise has linear cost $\OO(\#\TT_\ell)$, as does Dörfler marking with quasi-minimal cardinality; see~\cite{s2007} for $\const{C}{mark}=2$ and~\cite{pp2020} for $\const{C}{mark}=1$.
Moreover, refinement by newest-vertex bisection is known to admit an $\mathcal{O}(\#\TT_\ell)$ implementation; see, e.g.,~\cite{s2008, dgs2025}.
Overall, it is thus realistic to assume that the cumulative computational cost up to step $(\ell, k)\in\QQ$ is, up to a constant factor, given by
\begin{equation}\label{eq:costDef}
  \cost(\ell,k)\coloneq\sum_{\substack{(\ell^\prime, k^\prime)\in\QQ \\ \abs{\ell^\prime, k^\prime}\leq \abs{\ell, k}}}\#\TT_{\ell^\prime}\simeq \sum_{\substack{(\ell^\prime, k^\prime)\in\QQ \\ \abs{\ell^\prime, k^\prime}\leq \abs{\ell, k}}}\dim\XX_{\ell^\prime}.
\end{equation}

The following elementary result from~\cite{ghps2021,bfmps2025} highlights a central consequence of full R-linear convergence~\eqref{eq:linearConvergence}.
It establishes that convergence rates of $\Zeta_\ell^k$, measured in terms of the number of degrees of freedom $\dim\XX_\ell\simeq\#\TT_\ell$ and in terms of $\cost(\ell,k)$, agree.

\begin{corollary}[rates = complexity]\label{cor:ratesEqComplex}
  For all $s>0$ and $\Ccost[s]\coloneq\Clin(1-\qlin^{1/s})^{-s}$, full R-linear convergence~\eqref{eq:linearConvergence} guarantees
  \begin{equation*}
    \sup_{(\ell,k)\in\QQ}(\#\TT_\ell)^s\,\Zeta_\ell^k\leq\sup_{(\ell,k)\in\QQ}\cost(\ell,k)^s\,\Zeta_\ell^k
    \leq \Ccost[s]\sup_{(\ell,k)\in\QQ}(\#\TT_\ell)^s\,\Zeta_\ell^k. \tag*{\qed}
  \end{equation*}
\end{corollary}

The next result shows that, provided the adaptivity parameters $\lctr$ and $\theta$ are sufficiently small, a Dörfler criterion for the standard estimator $\eta_\ell$ from~\eqref{eq:residualEstimator} implies a Dörfler criterion for the reconstruction estimator $\zeta_\ell$ from~\eqref{eq:ellipticReconstructionEstimator}.
This is the key observation for establishing optimality of Algorithm~\ref{algorithm} steered by the reconstruction estimator $\zeta_\ell$.
The proof is given in section~\ref{subsec:DoerflerNecessary}.

\begin{lemma}[necessity of Dörfler marking]\label{lemma:doerflerNecessary}
  Suppose that the balancing parameter satisfies
  \begin{equation}\label{lemma:doerflerNecessary:lambdasmall}
    0 < \lctr < C[\qctr, \delta]^{-1}, \;\text{ where }\; C[\qctr, \delta]\coloneq (1+\delta)\,\Ceqv + \frac{\qctr\,\Cstabeta\,\delta}{1-\qctr}.
  \end{equation}
  Moreover, assume that, for some $0<\theta\leq1$ and $\RR_\ell\subseteq\TT_\ell$, the standard estimator $\eta_\ell$ from~\eqref{eq:residualEstimator} for the exact discrete solution $\starred{u}_\ell\in\XX_\ell$ fulfills the Dörfler criterion
  \begin{equation}\label{lemma:doerflerNecessary:premise}
    \thetamark\,\eta_\ell(\starred{u}_\ell)^2\leq \eta_\ell(\RR_\ell, \starred{u}_\ell)^2\quad\text{with}\quad 0 < \thetamark \coloneq \bigg[ \frac{\theta^{1/2} + C[\qctr, \delta] \, \lctr}{1 - C[\qctr, \delta]\, \lctr}\bigg]^2.
  \end{equation}
  Then, the reconstruction estimator $\zeta_\ell$ from~\eqref{eq:ellipticReconstructionEstimator} satisfies the Dörfler criterion
  \begin{equation}\label{lemma:doerflerNecessary:conclusion}
    \theta\,\zeta_\ell(u_\ell^{\kk-1}; \starred{z}_\ell^\kk)^{2} \leq \zeta_\ell(u_\ell^{\kk-1}; \RR_\ell, \starred{z}_\ell^\kk)^{2}.
  \end{equation}
\end{lemma}

With Lemma~\ref{lemma:doerflerNecessary} at hand, we are in the position to prove our second main result guaranteeing that Algorithm~\ref{algorithm} steered by the reconstruction estimator $\zeta_\ell$ decreases the quasi-error $\Zeta_\ell^k$ from~\eqref{eq:quasiError} at the optimal rate with respect to $\cost(\ell,k)$.
To formalize this, we recall the notion of approximation classes~\cite{BDDP02, bdd2004, s2007,ckns2008,axioms}.
For $N\in\N_0$, let $\T(N)$ denote the set of all refinements $\TT\in\T$ satisfying $\#\TT-\#\TT_0\leq N$.
Define
\begin{equation}\label{def:approxnorm}
  \| \starred{u} \|_{\mathbb{A}_s}\coloneq\sup_{N\in\N_0}\bigl[(N+1)^s\min_{\TT_{\rm opt}\in\T(N)}\eta_{\rm opt}(\starred{u}_{\rm opt})\bigr]\in [0,+\infty] \quad\text{ for all $s>0$}.
\end{equation}
The condition $\| \starred{u} \|_{\mathbb{A}_s}<+\infty$ is equivalent to a possible decay (at least) with algebraic rate $s>0$ of the estimator for the exact discrete solutions on optimal meshes.
The proof proceeds along the lines of~\cite[Theorem~8]{ghps2021} and is presented in section~\ref{subsec:rates}.

\begin{theorem}[optimal convergence rates and quasi-optimal complexity]\label{theorem:optimalRates}
  Recall $C[\qctr, \delta]$ from~\eqref{lemma:doerflerNecessary:lambdasmall} and $\thetamark$ from~\eqref{lemma:doerflerNecessary:premise}.
  If $0< \theta < 1$ and $\lctr>0$ satisfy
  \begin{equation}\label{theorem:optimalRates:assumptions}
    0 < \lctr < C[\qctr, \delta]^{-1}
    \quad\text{ and }\quad
    0
    <
    \thetamark
    \eqreff{lemma:doerflerNecessary:premise}{=}
    \bigg[ \frac{\theta^{1/2} + C[\qctr, \delta] \, \lctr}{1 - C[\qctr, \delta]\, \lctr}\bigg]^2
    <
    \thetaopt \coloneq\frac{1}{1 + \Cstab^2\, \Cdrel^2},
  \end{equation}
  then Algorithm~\ref{algorithm} steered by the elliptic reconstruction error estimator $\zeta_\ell$ converges with optimal rates with respect to the number of degrees of freedom $\dim\XX_\ell\simeq\#\TT_\ell$ in the sense that
  \begin{equation}\label{theorem:optimalRates:statement}
    \copt \norm{\starred{u}}_{\mathbb{A}_s}
    \le
    \sup_{(\ell,k) \in \QQ} \big( \#\TT_\ell\big)^s \, \Zeta_\ell^k
    \le
    \Copt \, \max\{\norm{\starred{u}}_{\mathbb{A}_s} , \Zeta_0^1\}
    \quad \text{for all } s>0
  \end{equation}
  and, by Corollary~\ref{cor:ratesEqComplex}, also with respect to the computational cost $\cost(\ell, k)$ defined in~\eqref{eq:costDef}.
  The constant $\copt$ depends only on properties of NVB, $\#\TT_0$, and $s$, while $\Copt$ additionally depends on $\Cmark, \theta, \lctr, \qctr, \delta, \Ceqv, \alpha, L$, and the constants in~\eqref{axiom:stability}--\eqref{axiom:discrete_reliability} and~\eqref{axiom:qm}.
\end{theorem}

\section{Proofs}\label{section:proofs}

\subsection{Proof of \texorpdfstring{Theorem~\ref{theorem:linearConvergence}}{Theorem 10}}
\label{subsec:proofUncond}

The proof is divided into six steps

  {\emph{Step~1.
    }} We show that there exist $0< q_\theta < 1$ and $C[\delta]>0$ such that
\begin{equation}\label{eq:step1:result}
  \eta_{\ell+1}(\starred{u}_{\ell+1})
  \le
  q_\theta \, \eta_\ell(\starred{u}_\ell) + C[\delta] \, \enorm{ \starred{u}_\ell - u_\ell^{\kk-1}} + \Cstabeta \, \enorm{ \starred{u}_{\ell+1} -  \starred{u}_\ell}\quad\text{ for all } 0 \le \ell < \underline\ell.
\end{equation}
Dörfler marking~\eqref{algorithm:doerfler}, estimator equivalence~\eqref{eq:estimatorEquivalence}, and stability~\eqref{axiom:stability} of $\eta_\ell$ imply that
\begin{equation}\label{eq:step1:auxiliary1}\begin{aligned}
    \theta^{1/2} \, \zeta_\ell(u_\ell^{\kk-1}; \starred{z}_\ell^\kk)
    \eqreff{algorithm:doerfler}{\leq}\zeta_\ell(u_\ell^{\kk-1}; \MM_\ell, \starred{z}_\ell^\kk)
     & \eqreff*{eq:estimatorEquivalence}{\:\leq\:}\eta_\ell(\MM_\ell, u_\ell^\kk)+(1+\delta)\,\Ceqv \, \enorm{\starred{z}_\ell^\kk}                                               \\
     & \eqreff*{axiom:stability}{\:\leq\:}\eta_\ell(\MM_\ell, \starred{u}_\ell)+\Cstabeta\,\enorm{\starred{u}_\ell-u_\ell^\kk}+(1+\delta)\,\Ceqv \, \enorm{\starred{z}_\ell^\kk}.
  \end{aligned}\end{equation}
For all $(\ell,k)\in\QQ$, the identity $\starred{z}_\ell^k=\delta^{-1}(u_\ell^k-u_\ell^{k-1})$, the estimate~\eqref{rem:convergence:eq1}, and the Zarantonello contraction~\eqref{eq:zarantonelloContraction} yield
\begin{equation}\label{eq:step1:auxiliary0}
  \enorm{\starred{z}_\ell^k}=\delta^{-1}\,\enorm{u_\ell^k-u_\ell^{k-1}}\eqreff{rem:convergence:eq1}{\leq} \delta^{-1}\,(1+\qctr)\,\enorm{\starred{u}_\ell-u_\ell^{k-1}}
  \quad\text{ and }\quad\enorm{\starred{u}_\ell-u_\ell^k}\eqreff{eq:zarantonelloContraction}{\leq}\qctr\, \enorm{\starred{u}_\ell-u_\ell^{k-1}}.
\end{equation}
Stability~\eqref{axiom:stability} of $\eta_\ell$ and estimator equivalence~\eqref{eq:estimatorEquivalence} together with~\eqref{eq:step1:auxiliary1}--\eqref{eq:step1:auxiliary0} show that
\begin{equation}\label{eq:step1:auxiliary2}\begin{aligned}
    \theta^{1/2} \, \eta_\ell(\starred{u}_\ell)
     & \eqreff*{axiom:stability}{\leqpad}\theta^{1/2} \, \eta_\ell(u_\ell^\kk) + \theta^{1/2}\,\Cstabeta\,\enorm{\starred{u}_\ell-u_\ell^\kk}                                                                                                     \\
     & \eqreff*{eq:estimatorEquivalence}{\leqpad}\theta^{1/2} \, \zeta_\ell(u_\ell^{\kk-1}; \starred{z}_\ell^\kk) + \theta^{1/2} \,(1+\delta)\,\Ceqv\,\enorm{\starred{z}_\ell^\kk} + \theta^{1/2}\,\Cstabeta\,\enorm{\starred{u}_\ell-u_\ell^\kk} \\
     & \eqreff*{eq:step1:auxiliary1}{\leqpad}\eta_\ell(\MM_\ell, \starred{u}_\ell) + (1+\theta^{1/2}) \,(1+\delta)\,\Ceqv\,\enorm{\starred{z}_\ell^\kk} + (1+\theta^{1/2})\,\Cstabeta\,\enorm{\starred{u}_\ell-u_\ell^\kk}                        \\
     & \eqreff*{eq:step1:auxiliary0}{\leqpad}\eta_\ell(\MM_\ell, \starred{u}_\ell)
    +(1+\theta^{1/2})\biggl[\frac{(1+\delta)\,\Ceqv\,(1+\qctr)}{\delta} + \Cstabeta\,\qctr \biggr]\,\enorm{\starred{u}_\ell-u_\ell^{\kk-1}}.
  \end{aligned}\end{equation}
Define $\widetilde{C}[\delta]\coloneq (1+\theta^{1/2})[(1+\delta)\,\Ceqv\,(1+\qctr)\,\delta^{-1} + \Cstabeta\,\qctr]$.
Since all marked elements are refined, i.e., $\MM_\ell\subseteq\TT_\ell\setminus\TT_{\ell+1}$, squaring and applying~\eqref{eq:step1:auxiliary2} yields
\begin{equation}\label{eq:step1:auxiliary3}
  \frac{\theta}{2} \, \eta_\ell(\starred{u}_\ell)^2\eqreff{eq:step1:auxiliary2}{\leq}
  \eta_\ell(\TT_\ell\setminus\TT_{\ell+1}, \starred{u}_\ell)^2 + \widetilde{C}[\delta]^2 \, \enorm{\starred{u}_\ell-u_\ell^{\kk-1}}^2.
\end{equation}
Using stability~\eqref{axiom:stability} and reduction~\eqref{axiom:reduction} for $\eta_\ell$, estimate~\eqref{eq:step1:auxiliary3} validates
\begin{equation}\label{eq:step1:auxiliary4}\begin{aligned}
    \eta_{\ell+1}(\starred{u}_{\ell})^2
     & \leqpad
    \eta_{\ell}(\starred{u}_{\ell})^2 - (1-\qred^2) \, \eta_{\ell}(\TT_\ell \setminus \TT_{\ell+1}, \starred{u}_{\ell})^2 \\
     & \eqreff*{eq:step1:auxiliary3}\leqpad
    \Big( 1- (1-\qred^2) \, \frac{\theta}{2} \Big) \, \eta_{\ell}(\starred{u}_{\ell})^2
    + (1-\qred^2)\widetilde{C}[\delta]^2 \, \enorm{\starred{u}_\ell-u_\ell^{\kk-1}}^2.
  \end{aligned}\end{equation}
For $q_\theta \coloneqq [1 - (1-\qred^2) \, \theta/2]^{1/2}<1$ and $C[\delta]\coloneqq (1-\qred^2)^{1/2}\,\widetilde{C}[\delta]$, estimate~\eqref{eq:step1:auxiliary4} together with stability~\eqref{axiom:stability} of $\eta$ concludes the proof of~\eqref{eq:step1:result} via
\begin{equation*}\begin{aligned}
    \eta_{\ell+1}(\starred{u}_{\ell+1})
     & \eqreff*{axiom:stability}{\leqpad}\eta_{\ell+1}(\starred{u}_{\ell})+\Cstabeta\,\enorm{\starred{u}_{\ell+1}-\starred{u}_{\ell}}                                              \\
     & \eqreff*{eq:step1:auxiliary4}\leqpad\big[q_\theta^2 \, \eta_\ell( \starred{u}_\ell)^2 + C[\delta]^2 \, \enorm{ \starred{u}_\ell - u_\ell^{\kk-1}}^2\big]^{1/2}
    + \Cstabeta \, \enorm{ \starred{u}_{\ell+1} -  \starred{u}_\ell}                                                                                                               \\
     & \leqpad q_\theta \, \eta_\ell(\starred{u}_\ell) + C[\delta] \, \enorm{ \starred{u}_\ell - u_\ell^{\kk-1}} + \Cstabeta \, \enorm{ \starred{u}_{\ell+1} -  \starred{u}_\ell}.
  \end{aligned}\end{equation*}

{\emph{Step~2.}}
For $\gamma>0$ such that $\qctr + \gamma\,C[\delta]<1$, define the auxiliary quasi-error
\begin{equation}\label{eq:quasiError:meshLevel}
  \Lambda_\ell
  \coloneqq
  \enorm{ \starred{u}_\ell - u_{\ell}^{\underline k-1}}
  +
  \gamma \, \eta_\ell( \starred{u}_\ell)
  \quad \text{ for all } 0 \le \ell < \underline\ell.
\end{equation}
The Zarantonello contraction~\eqref{eq:zarantonelloContraction} and nested iteration $u_{\ell+1}^0 = u_\ell^{\underline k}$ imply that
\begin{equation}\label{eq:step2:auxiliary1}
  \enorm{ \starred{u}_{\ell+1} - u_{\ell+1}^{\underline k-1}} \,
  \eqreff{eq:zarantonelloContraction}{\le}  \,
  \enorm{ \starred{u}_{\ell+1} - u_{\ell+1}^{0}}
  =
  \enorm{ \starred{u}_{\ell+1} - u_{\ell}^{\underline k}} \\
  \eqreff{eq:zarantonelloContraction}{\le}
  \qctr \, \enorm{\starred{u}_\ell - u_{\ell}^{\underline k -1}}+\enorm{ \starred{u}_{\ell+1} - \starred{u}_\ell}.
\end{equation}
The combination of the last inequality~\eqref{eq:step2:auxiliary1} and~\eqref{eq:step1:result} from Step~1 results in
\begin{equation}\begin{aligned}
    \Lambda_{\ell+1}
     & \eqreff*{eq:quasiError:meshLevel}{\eqpad} \enorm{ \starred{u}_{\ell+1} - u_{\ell+1}^{\underline k-1}} + \gamma \, \eta_{\ell+1}( \starred{u}_{\ell+1}) \\
     & \eqreff*{eq:step2:auxiliary1}{\leqpad}
    \qctr \, \enorm{\starred{u}_\ell - u_{\ell}^{\underline k -1}}
    +
    \enorm{ \starred{u}_{\ell+1} - \starred{u}_\ell}
    + \gamma \, \eta_{\ell+1}( \starred{u}_{\ell+1})                                                                                                          \\
     & \eqreff*{eq:step1:result}\leqpad
    (\qctr + \gamma \, C[\delta]) \, \enorm{\starred{u}_\ell - u_{\ell}^{\underline k -1}}
    + q_\theta \, \gamma \, \eta_\ell( \starred{u}_\ell) +
    (1+ \gamma \, \Cstabeta)\, \enorm{ \starred{u}_{\ell+1} - \starred{u}_\ell}                                                                               \\
     & \leqpad
    \max\{\qctr + \gamma \, C[\delta] ,q_\theta\} \,  \Lambda_\ell +
    (1+ \gamma \, \Cstabeta)\, \enorm{ \starred{u}_{\ell+1} - \starred{u}_\ell}.
  \end{aligned}\end{equation}
The assumption on $\gamma$ ensures that $0 < \max\{\qctr + \gamma \, C[\delta] ,q_\theta\} < 1$.
Recall the energy functional $\EE$ from~\eqref{def:energy}.
The equivalence~\eqref{eq:energyEqv} together with a telescoping sum and reliability~\eqref{axiom:reliability} shows that
\begin{equation*}\begin{aligned}
    &\frac{\alpha}{2}\,\sum_{\ell^\prime=\ell}^{\ell+N}  \enorm{\starred{u}_{\ell^\prime+1}-\starred{u}_{\ell^\prime}}^2
      \eqreff{eq:energyEqv}{\leq} \sum_{\ell^\prime=\ell}^{\ell+N}\big[\EE(\starred{u}_{\ell^\prime})-\EE(\starred{u}_{\ell^\prime+1})\big]
    =\EE(\starred{u}_\ell)-\EE(\starred{u}_{\ell+N+1})
    \leq\EE(\starred{u}_\ell)-\EE(\starred{u})\\
    &\eqreff{eq:energyEqv}{\leq}\frac{L}{2}\,\enorm{\starred{u} - \starred{u}_\ell}^2
    \eqreff{axiom:reliability}{\leq}\frac{L\,\Creleta^2}{2}\,\eta_\ell(\starred{u}_\ell)^2\eqreff{eq:quasiError:meshLevel}{\leq}\frac{L\,\Creleta^2}{2\gamma^2}\,\Lambda_\ell^2\quad\text{ for all $\ell,N\in\N_0$ with $\ell\leq\ell+N<\ellu$.}
  \end{aligned}\end{equation*}
Consequently, $\Lambda_\ell$ (extended by zero if $\ellu<\infty$) satisfies the assumptions of the tail summability criterion~\cite[Lemma~10]{BLP26arXiv}.
Thus, there exists a constant $\Caux>0$ such that
\begin{equation}\label{eq:step2:claim}
  \sum_{\ell^\prime=\ell+1}^{\ellu-1}\Lambda_{\ell^\prime} \leq \Caux\,\Lambda_\ell \quad \text{ for all } 0 \le \ell < \underline\ell,
\end{equation}
where $\Caux$ depends only on $\qctr, \theta, \delta, \Ceqv, \alpha, L$, and the constants in~\eqref{axiom:stability}--\eqref{axiom:discrete_reliability}.

{\emph{Step~3.}}
Recall the quasi-errors $\Zeta_\ell^{k}$ from~\eqref{eq:quasiError} and $ \Lambda_{\ell}$ from~\eqref{eq:quasiError:meshLevel}.
This step shows that
\begin{equation}\label{eq:step3:estimatorEquivalence}
  \Zeta_\ell^{\underline k} \simeq  \Lambda_{\ell}\quad\text{ for all $0\leq\ell<\ellu$.}
\end{equation}
For $(\ell, k)\in\QQ$, the estimator equivalence~\eqref{eq:estimatorEquivalence}, stability~\eqref{axiom:stability} of $\eta_\ell$, and~\eqref{eq:step1:auxiliary0} imply
\begin{equation}\label{eq:step3:aux1}\begin{aligned}
    \zeta_\ell(u_\ell^{k-1}; \starred{z}_\ell^k)
     & \eqreff*{eq:estimatorEquivalence}{\leqpad}\eta_\ell(u_\ell^k) + (1+\delta)\, \Ceqv \, \enorm{\starred{z}_\ell^k}                                                                       \\
     & \eqreff*{axiom:stability}{\leqpad}\eta_\ell(\starred{u}_\ell) + \Cstabeta\,\enorm{\starred{u}_\ell - u_\ell^k}+ (1+\delta)\, \Ceqv \, \enorm{\starred{z}_\ell^k}                       \\
     & \eqreff*{eq:step1:auxiliary0}{\leqpad} \eta_\ell(\starred{u}_\ell) + \bigl[ \Cstabeta\,\qctr + (1+\delta)\,\Ceqv\,\delta^{-1}\,(1+\qctr)\bigr]\enorm{\starred{u}_\ell - u_\ell^{k-1}}.
  \end{aligned}\end{equation}
An analogous argument shows that
\begin{equation}\label{eq:step3:aux2}
  \eta_\ell(\starred{u}_\ell)\leq \zeta_\ell(u_\ell^{k-1}; \starred{z}_\ell^k)+ \bigl[ \Cstabeta\,\qctr + (1+\delta)\,\Ceqv\,\delta^{-1}\,(1+\qctr)\bigr]\enorm{\starred{u}_\ell - u_\ell^{k-1}}.
\end{equation}
Overall, the estimates~\eqref{eq:step3:aux1}--\eqref{eq:step3:aux2} with $k=\kk[\ell]$ establish~\eqref{eq:step3:estimatorEquivalence}.

{\emph{Step~4.}}
This step proves tail summability of $\Zeta_\ell^k$ in $k$, i.e.,
\begin{equation}\label{eq:step4:claim}
  \Zeta_\ell^{k'}
  \lesssim
  \qctr^{k'-k} \, \Zeta_\ell^k
  \quad \text{ for all } 1 \le k \le k' \le \underline k[\ell].
\end{equation}
Distinguish between two cases.
First, let $1 \le k \le k' < \underline k[\ell]$.
Since~\eqref{algorithm:stoppingCriterion} is not met, it follows that
\begin{equation}\label{eq:step4:aux1}\begin{aligned}
    \Zeta_\ell^{k'}
     & \eqreff*{eq:quasiError}{\eqpad}
    \enorm{\starred{u}_\ell - u_\ell^{k'-1}} + \zeta_\ell(u_\ell^{k'-1}; \starred{z}_\ell^{k'})
    \\
     & \eqreff*{algorithm:stoppingCriterion}{\lepad}
    \enorm{\starred{u}_\ell - u_\ell^{k'-1}} + \lctr^{-1} \, \enorm{\starred{z}_\ell^{k'}}                                                        \\
     & \eqreff{eq:step1:auxiliary0}{\leqpad} \bigl[1 + \lctr^{-1}\,\delta^{-1}\,(1+\qctr)\bigr] \, \enorm{\starred{u}_\ell - u_\ell^{k'-1}}       \\
     & \eqreff{eq:zarantonelloContraction}{\leqpad} \bigl[1 + \lctr^{-1}\,\delta^{-1}\,(1+\qctr)\bigr]\qctr^{k'-k}\, \enorm{\starred{u}_\ell - u_\ell^{k-1}}
    \;\lesssim\; \qctr^{k'-k} \, \Zeta_\ell^k.
  \end{aligned}\end{equation}
Second, it remains to consider the case $k^\prime=\underline k[\ell]$.
If $\kk[\ell]=1$, the claim~\eqref{eq:step4:claim} is evident.
For $k^\prime=\kk[\ell]\geq 2$,
the estimator equivalence~\eqref{eq:estimatorEquivalence}, the stability~\eqref{axiom:stability} of $\eta$, and~\eqref{eq:step1:auxiliary0} show that
\begin{equation}\label{eq:step4:aux2}\begin{aligned}
    \zeta_\ell(u_\ell^{\underline k-1}; \starred{z}_\ell^\kk)
     & \eqreff*{eq:estimatorEquivalence}{\leqpad}
    \eta_\ell(u_\ell^\kk) + (1+\delta)\,\Ceqv\,\enorm{\starred{z}_\ell^\kk}                                                                                                                                                                                 \\
     & \eqreff*{axiom:stability}{\leqpad} \eta_\ell(u_\ell^{\kk-1}) + \Cstabeta\,\enorm{u_\ell^\kk-u_\ell^{\kk-1}} + (1+\delta)\,\Ceqv\,\enorm{\starred{z}_\ell^\kk}                                                                                        \\
     & \eqreff*{eq:estimatorEquivalence}{\leqpad} \zeta_\ell(u_\ell^{\underline k-2}; \starred{z}_\ell^{\kk-1}) + \Cstabeta\,\enorm{u_\ell^\kk-u_\ell^{\kk-1}} + (1+\delta)\,\Ceqv\,\big(\enorm{\starred{z}_\ell^{\kk-1}}+\enorm{\starred{z}_\ell^\kk}\big) \\
     & \eqreff*{eq:step1:auxiliary0}\lesssimpad
    \zeta_\ell(u_\ell^{\underline k-2}; \starred{z}_\ell^{\kk-1})+\enorm{ \starred{u}_\ell - u_\ell^{\underline k -1}} +\enorm{ \starred{u}_\ell - u_\ell^{\underline k -2}}.
  \end{aligned}\end{equation}
Estimate~\eqref{eq:step4:aux2} together with the Zarantonello contraction~\eqref{eq:zarantonelloContraction} establishes stability of $\Zeta_\ell^{\underline k}$ via
\begin{equation}\label{eq:step4:aux3}
  \Zeta_\ell^{\underline k}
  \eqreff*{eq:quasiError}{\eqpad}
  \enorm{ \starred{u}_\ell - u_\ell^{\underline k -1}}
  + \zeta_\ell(u_\ell^{\underline k-1}; \starred{z}_\ell^\kk)
  \eqreff{eq:step4:aux2}{\lesssim} \enorm{ \starred{u}_\ell - u_\ell^{\underline k -1}} +\zeta_\ell(u_\ell^{\underline k-2}; \starred{z}_\ell^{\kk-1})+\enorm{ \starred{u}_\ell - u_\ell^{\underline k -2}}
  \eqreff{eq:zarantonelloContraction}{\lesssim}\Zeta_\ell^{\underline k -1}.
\end{equation}
At the expense of $1/\qctr$ in the multiplicative constant,~\eqref{eq:step4:aux1} and~\eqref{eq:step4:aux3} yield~\eqref{eq:step4:claim} in all cases.

  {\emph{Step~5.}}
Reliability~\eqref{axiom:reliability}, quasi-monotonicity~\eqref{axiom:qm} of $\eta$, and the contraction~\eqref{eq:zarantonelloContraction} yield
\begin{equation}\label{eq:step5:auxiliary}\begin{aligned}
    \enorm{ \starred{u}_{\ell+1} - u_{\ell}^{\underline k}} &
    \le
    \enorm{\starred{u}_\ell - u_{\ell}^{\underline k}} + \enorm{ \starred{u}_{\ell+1} - \starred{u}_\ell}
    \eqreff{axiom:reliability}{\lesssim}
    \enorm{\starred{u}_\ell - u_{\ell}^{\underline k}}
    +
    \eta_\ell( \starred{u}_\ell)+\eta_{\ell+1}( \starred{u}_{\ell+1})                          \\
                                                            & \quad\eqreff{axiom:qm}{\lesssim}
    \enorm{\starred{u}_\ell - u_{\ell}^{\underline k}}
    +
    \eta_\ell( \starred{u}_\ell)
    \eqreff{eq:zarantonelloContraction}{\leq}
    \eta_\ell( \starred{u}_\ell)
    + \qctr\, \enorm{\starred{u}_\ell - u_{\ell}^{\underline k -1}}.
  \end{aligned}\end{equation}
Applying~\eqref{eq:step3:aux1} with $k=1$ and~\eqref{eq:step3:aux2} with $k=\kk[\ell]$, together with quasi-monotonicity~\eqref{axiom:qm} of $\eta$ and nested iteration $u_\ell^{\underline k} = u_{\ell+1}^0$, yields stability $\Zeta_\ell^k$ under mesh refinement, i.e.,
\begin{equation}\label{eq:step5:claim}\begin{aligned}
    \Zeta_{\ell+1}^1
    \eqreff{eq:quasiError}= &
    \enorm{ \starred{u}_{\ell+1} - u_{\ell+1}^{0}}
    + \zeta_{\ell+1}(u_{\ell+1}^{0}; \starred{z}_{\ell+1}^1)
    \eqreff{eq:step3:aux1}\lesssim
    \enorm{ \starred{u}_{\ell+1} - u_{\ell+1}^{0}}
    + \eta_{\ell+1}( \starred{u}_{\ell+1}) \\
                            &
    \eqreff{axiom:qm}\lesssim
    \enorm{ \starred{u}_{\ell+1} - u_{\ell}^{\underline k}}
    + \eta_{\ell}(\starred{u}_\ell)
    \eqreff{eq:step5:auxiliary}\lesssim
    \enorm{\starred{u}_\ell - u_{\ell}^{\underline k -1}}
    + \eta_{\ell}(\starred{u}_\ell)
    \eqreff{eq:step3:aux2}\lesssim
    \Zeta_\ell^{\underline k}.
  \end{aligned}\end{equation}

{\emph{Step~6.}}
Finally, tail summability of $\Zeta_\ell^k$ in both $\ell$ and $k$ follows from Steps~2--5 via
\begin{equation*}
  \begin{aligned}
     & \sum_{\substack{(\ell',k') \in \QQ                     \\ |\ell',k'| > |\ell,k|}}
    \Zeta_{\ell'}^{k'}
    =
    \sum_{k' = k+1}^{\underline k[\ell]} \Zeta_\ell^{k'}
    +
    \sum_{\ell' = \ell + 1}^{\underline \ell} \sum_{k' = 1}^{\underline k[\ell^\prime]} \Zeta_{\ell'}^{k'}
    \eqreff{eq:step4:claim}{\lesssim}
    \Zeta_\ell^k
    +
    \sum_{\ell' = \ell + 1}^{\underline \ell} \Zeta_{\ell'}^1
    =
    \Zeta_\ell^k
    +
    \sum_{\ell' = \ell}^{\underline \ell-1} \Zeta_{\ell'+1}^1 \\
     & \quad\eqreff{eq:step5:claim}{\lesssim}
    \Zeta_\ell^k
    +
    \sum_{\ell' = \ell}^{\underline \ell-1} \Zeta_{\ell'}^{\underline k}
    \eqreff{eq:step3:estimatorEquivalence}{\simeq}
    \Zeta_\ell^k
    +
    \sum_{\ell' = \ell}^{\underline \ell-1}  \Lambda_{\ell'}
    \eqreff{eq:step2:claim}\lesssim
    \Zeta_\ell^k
    +
    \Lambda_{\ell}
    \eqreff{eq:step3:estimatorEquivalence}\simeq
    \Zeta_\ell^k + \Zeta_{\ell}^{\underline k}
    \eqreff{eq:step4:claim}\lesssim
    \Zeta_{\ell}^{k}.
  \end{aligned}
\end{equation*}
The hidden constant depends only on $\lambda, \qctr, \theta, \delta, \Ceqv, \alpha, L$, and the constants in~\eqref{axiom:stability}--\eqref{axiom:discrete_reliability} and~\eqref{axiom:qm}.
Therefore,~\cite[Lemma~2]{bfmps2025} concludes the proof of full R-linear convergence~\eqref{eq:linearConvergence}; see also~\cite[Lemma~4.9]{axioms}, where the summation argument appeared first.
\qed

\subsection{Proof of \texorpdfstring{Lemma~\ref{lemma:doerflerNecessary}}{Lemma 14}}
\label{subsec:DoerflerNecessary}

The argument used to derive~\eqref{eq:APEstim} also applies to the discrete Zarantonello map $\Phi_\ell$ and yields for $u_\ell^\kk=\Phi_\ell(u_\ell^{\kk-1})$ that
\begin{equation}\label{lem:dn:aux1}
  \enorm{\starred{u}_\ell-u_\ell^\kk}\eqreff{eq:zarantonelloContraction}{\leq}\frac{\qctr}{1-\qctr}\enorm{u_\ell^\kk-u_\ell^{\kk-1}}=\frac{\qctr\,\delta}{1-\qctr}\enorm{\starred{z}_\ell^\kk}.
\end{equation}
From stability~\eqref{axiom:stability} of $\eta_\ell$, estimator equivalence~\eqref{eq:estimatorEquivalence}, estimate~\eqref{lem:dn:aux1}, and the met stopping criterion~\eqref{algorithm:stoppingCriterion}, it follows that
\begin{equation}\label{lem:dn:aux2}\begin{aligned}
    \eta_\ell(\RR_\ell, \starred{u}_\ell)
     & \eqreff*{axiom:stability}{\leqpad}\eta_\ell(\RR_\ell, u_\ell^\kk) + \Cstabeta\,\enorm{\starred{u}_\ell-u_\ell^\kk}                                                                                                                      \\
     & \eqreff*{eq:estimatorEquivalence}{\leqpad}\zeta_\ell(u_{\ell}^{\kk -1};\RR_\ell, \starred{z}_\ell^\kk)+(1+\delta)\,\Ceqv\,\enorm{\starred{z}_\ell^\kk}+ \Cstabeta\,\enorm{\starred{u}_\ell-u_\ell^\kk}                                  \\
     & \eqreff*{lem:dn:aux1}{\leqpad}\zeta_\ell(u_{\ell}^{\kk -1};\RR_\ell, \starred{z}_\ell^\kk)+\Bigl[ (1+\delta)\,\Ceqv + \frac{\qctr\,\Cstabeta\,\delta}{1-\qctr}\Bigr]\enorm{\starred{z}_\ell^\kk}                                        \\
     & \eqreff*{algorithm:stoppingCriterion}{\leqpad}\zeta_\ell(u_{\ell}^{\kk -1};\RR_\ell, \starred{z}_\ell^\kk)+\Bigl[ (1+\delta)\,\Ceqv + \frac{\qctr\,\Cstabeta\,\delta}{1-\qctr}\Bigr]\lctr\,\zeta(u_\ell^{\kk-1}; \starred{z}_\ell^\kk).
  \end{aligned}\end{equation}
Define $C[\qctr, \delta] \coloneq (1+\delta)\,\Ceqv + \qctr\,(1-\qctr)^{-1}\,\Cstabeta\,\delta$.
The analogous argument shows that
\begin{equation}\label{lem:dn:aux3}
  \zeta_\ell(u_\ell^{\kk-1}; \RR_\ell, \starred{z}_\ell^\kk) \leq \eta_\ell(\RR_\ell, \starred{u}_\ell) + C[\qctr, \delta]\,\lambda\, \zeta_\ell(u_\ell^{\kk-1}; \starred{z}_\ell^\kk).
\end{equation}
By~\eqref{lemma:doerflerNecessary:premise} and~\eqref{lem:dn:aux2}, there holds that
\begin{equation}\label{lem:dn:aux4}
  \thetamark^{1/2}\,\eta_\ell(\starred{u}_\ell)
  \eqreff{lemma:doerflerNecessary:premise}{\leq}\eta_\ell(\RR_\ell, \starred{u}_\ell)
  \eqreff{lem:dn:aux2}{\leq}\zeta_\ell(u_{\ell}^{\kk -1};\RR_\ell, \starred{z}_\ell^\kk)+C[\qctr, \delta]\,\lctr\,\zeta(u_\ell^{\kk-1}; \starred{z}_\ell^\kk).
\end{equation}
Applying~\eqref{lem:dn:aux3} with $\RR_\ell = \TT_\ell$, and using the smallness of $\lctr$ from~\eqref{lemma:doerflerNecessary:lambdasmall} and~\eqref{lem:dn:aux4}, it follows that
\begin{equation}\label{lem:dn:aux5}\begin{aligned}
    \thetamark^{1/2}\, \zeta_\ell(u_\ell^{\kk-1}; \starred{z}_\ell^\kk)
     & \eqreff*{lem:dn:aux3}{\leqpad}\frac{\thetamark^{1/2}}{1-C[\qctr, \delta]\,\lambda}\, \eta_\ell(\starred{u}_\ell)                                                                                               \\
     & \eqreff*{lem:dn:aux4}{\leqpad} \frac{1}{1-C[\qctr, \delta]\,\lambda}\,\bigl[\zeta_\ell(u_{\ell}^{\kk -1};\RR_\ell, \starred{z}_\ell^\kk)+C[\qctr, \delta]\,\lctr\,\zeta(u_\ell^{\kk-1}; \starred{z}_\ell^\kk)\bigr].
  \end{aligned}\end{equation}
Finally, recalling the definition of $\thetamark$ from~\eqref{lemma:doerflerNecessary:premise}, estimate~\eqref{lem:dn:aux5} concludes the proof of~\eqref{lemma:doerflerNecessary:conclusion} via
\begin{equation*}
  \theta^{1/2}\,\zeta_\ell(u_\ell^{\kk-1}; \starred{z}_\ell^\kk)
  \eqreff{lemma:doerflerNecessary:premise}{=} \big[ \big( 1- C[\qctr, \delta]\,\lctr \big) \, \thetamark^{1/2} - C[\qctr, \delta] \, \lctr \big]\,\zeta_\ell(u_\ell^{\kk-1}; \starred{z}_\ell^\kk)
  \eqreff{lem:dn:aux5}{\leq}\zeta_\ell(u_\ell^{\kk-1}; \RR_\ell, \starred{z}_\ell^\kk). \tag*{\qed}
\end{equation*}

\subsection{Proof of \texorpdfstring{Theorem~\ref{theorem:optimalRates}}{Theorem 14}}
\label{subsec:rates}

The argument of~\cite[Lemma~14]{ghps2021} together with the equivalence~\eqref{rem:condConv:eq1} from Remark~\ref{rem:condConv} yields the lower estimate in~\eqref{theorem:optimalRates:statement}.
For the upper estimate in~\eqref{theorem:optimalRates:statement}, assume without loss of generality that $\norm{\starred{u}}_{\mathbb{A}_s}<+\infty$.
The proof proceeds in three steps.

  {\emph{Step~1.}}
For all $0\leq\ell^\prime<\ellu$, we apply~\cite[Lemma~14]{ghps2021} with $\thetamark$ from~\eqref{theorem:optimalRates:assumptions}.
This yields a set $\RR_{\ell^\prime}\subseteq\TT_{\ell^\prime}$ and constants $\const{C}{cmp}, \const{\widetilde{C}}{cmp}>0$ depending only on $\Cstabeta, \Cdreleta, \Cmoneta$ such that
\begin{equation}\label{eq:rates:aux1}
  \thetamark\,\eta_{\ell^\prime}(\starred{u}_{\ell^\prime})^{2}\leq\eta_{\ell^\prime}(\RR_{\ell^\prime}, \starred{u}_{\ell^\prime})^{2}
  \quad\text{ and }\quad
  \#\RR_{\ell^\prime}\leq \const{\widetilde{C}}{cmp}\,\const{C}{cmp}^{1/s}\, \norm{\starred{u}}_{\mathbb{A}_s}^{1/s}\,\eta_{\ell^\prime}(\starred{u}_{\ell^\prime})^{-1/s}.
\end{equation}
Lemma~\ref{lemma:doerflerNecessary} implies that the reconstruction estimator $\zeta_\ell$ satisfies the Dörfler criterion~\eqref{lemma:doerflerNecessary:conclusion} with $\RR_{\ell^\prime}$.
Since Algorithm~\ref{algorithm} determines $\MM_{\ell^\prime}$ in~\eqref{algorithm:doerfler} with quasi-minimal cardinality, it follows that
\begin{equation}\label{eq:rates:aux2}
  \#\MM_{\ell^\prime}\leq\Cmark\,\#\RR_{\ell^\prime}\eqreff{eq:rates:aux1}{\lesssim}\norm{\starred{u}}_{\mathbb{A}_s}^{1/s}\,\eta_{\ell^\prime}(\starred{u}_{\ell^\prime})^{-1/s}.
\end{equation}
Moreover, estimate~\eqref{rem:convergence:eq1}, the met stopping criterion~\eqref{algorithm:stoppingCriterion}, and applying~\eqref{lem:dn:aux3} with $\RR_\ell = \TT_\ell$ (using the smallness of $\lctr$ from~\eqref{theorem:optimalRates:assumptions}) show that
\begin{equation}\label{eq:rates:aux3}
  \Zeta_{\ell^\prime}^\kk
  \eqreff{eq:quasiError}{=}\enorm{\starred{u}_{\ell^\prime}-u_{\ell^\prime}^{\kk-1}}+\zeta(u_{\ell^\prime}^{\kk-1}; \starred{z}_{\ell^\prime}^\kk)
  \eqreff{rem:convergence:eq1}{\lesssim}\enorm{\starred{z}_{\ell^\prime}^\kk}+\zeta(u_{\ell^\prime}^{\kk-1}; \starred{z}_{\ell^\prime}^\kk)
  \eqreff{algorithm:stoppingCriterion}{\lesssim}\zeta(u_{\ell^\prime}^{\kk-1}; \starred{z}_{\ell^\prime}^\kk)
  \eqreff{lem:dn:aux3}{\lesssim}\eta_{\ell^\prime}(\starred{u}_{\ell^\prime}).
\end{equation}
Together with~\eqref{eq:rates:aux2} and the stability~\eqref{eq:step5:claim} of $\Zeta_{\ell^\prime}^\kk$ under mesh refinement, this results in
\begin{equation}\label{eq:rates:result1}
  \#\MM_{\ell^\prime}
  \eqreff{eq:rates:aux2}{\lesssim}\norm{\starred{u}}_{\mathbb{A}_s}^{1/s}\,\eta_{\ell^\prime}(\starred{u}_{\ell^\prime})^{-1/s}
  \eqreff{eq:rates:aux3}{\lesssim}\norm{\starred{u}}_{\mathbb{A}_s}^{1/s}\,(\Zeta_{\ell^\prime}^\kk)^{-1/s}
  \eqreff{eq:step5:claim}{\lesssim}\norm{\starred{u}}_{\mathbb{A}_s}^{1/s}\,(\Zeta_{\ell^\prime+1}^{1})^{-1/s}.
\end{equation}

{\emph{Step~2.}}
Let $(\ell,k)\in\QQ$ with $\TT_\ell\neq\TT_0$.
Combining $\#\TT_0<\#\TT_\ell$, the mesh-closure estimate for NVB triangulations~\cite{bdd2004, s2008,affkp2013,dgs2025}, and~\eqref{eq:rates:result1} implies that
\begin{equation}\label{eq:rates:aux4}
  \#\TT_\ell-\#\TT_0+1
  \leq 2(\#\TT_\ell-\#\TT_0)
  \lesssim\sum_{\ell^\prime=0}^{\ell-1}\#\MM_{\ell^\prime}
  \eqreff{eq:rates:result1}{\lesssim}\norm{\starred{u}}_{\mathbb{A}_s}^{1/s}\sum_{\ell^\prime=0}^{\ell-1}(\Zeta_{\ell^\prime+1}^{1})^{-1/s}.
\end{equation}
Using~\eqref{eq:rates:aux4} and full R-linear convergence~\eqref{eq:linearConvergence} together with a geometric series argument yields
\begin{equation*}
  \#\TT_\ell-\#\TT_0+1
  \eqreff{eq:rates:aux4}{\lesssim}\norm{\starred{u}}_{\mathbb{A}_s}^{1/s}\sum_{\ell^\prime=0}^{\ell-1}(\Zeta_{\ell^\prime+1}^{1})^{-1/s}
  \leq \norm{\starred{u}}_{\mathbb{A}_s}^{1/s}\sum_{\substack{(\ell^\prime, k^\prime)\in\QQ \\ \abs{\ell^\prime, k^\prime}\leq \abs{\ell, k}}}(\Zeta_{\ell^\prime}^{k^\prime})^{-1/s}
  \eqreff{eq:linearConvergence}{\lesssim}\norm{\starred{u}}_{\mathbb{A}_s}^{1/s}\,(\Zeta_\ell^k)^{-1/s}.
\end{equation*}
Rearranging the terms, the previous formula proves that
\begin{equation}\label{eq:rates:result2}
  (\#\TT_\ell-\#\TT_0+1)^s\,\Zeta_\ell^k\lesssim\norm{\starred{u}}_{\mathbb{A}_s}
  \quad\text{ for all $(\ell,k)\in\QQ$ with $\TT_\ell\neq\TT_0$.}
\end{equation}

{\emph{Step~3.}}
Let $(\ell,k)\in\QQ$ with $\TT_\ell=\TT_0$ and $\ell>0$.
Then, $\MM_0=\emptyset$ and thus $\zeta_0(u_0^{\kk-1}; \starred{z}_0^\kk)=0$ by~\eqref{algorithm:doerfler}.
The stopping criterion~\eqref{algorithm:stoppingCriterion} therefore implies $\starred{z}_0^\kk=0$ and hence the estimate~\eqref{lem:dn:aux1} yields $u_0^\kk=u_0^{\kk-1}=\starred{u}_0$.
Consequently, $\Zeta_0^\kk=0$ and by R-linear convergence~\eqref{eq:linearConvergence} also $\Zeta_\ell^k=0$.
This shows that~\eqref{eq:rates:result2} extends to
\begin{equation}\label{eq:rates:aux5}
  (\#\TT_\ell-\#\TT_0+1)^s\,\Zeta_\ell^k\lesssim\norm{\starred{u}}_{\mathbb{A}_s}
  \quad\text{ for all $(\ell,k)\in\QQ$ with $0<\ell$.}
\end{equation}
Since R-linear convergence~\eqref{eq:linearConvergence} ensures $\Zeta_0^k\lesssim\Zeta_0^1$ for all $1\leq k\leq\kk[0]$, it follows from~\eqref{eq:rates:aux5} that
\begin{equation}\label{eq:rates:aux6}
  (\#\TT_\ell-\#\TT_0+1)^s\,\Zeta_\ell^k\lesssim\max\{\norm{\starred{u}}_{\mathbb{A}_s}, \Zeta_0^1\}
  \quad\text{ for all $(\ell,k)\in\QQ$.}
\end{equation}
Finally,~\cite[Lemma~22]{bhp2017} shows that $\#\TT_\ell-\#\TT_0+1\simeq\#\TT_\ell$.
Together with~\eqref{eq:rates:aux6}, this concludes the proof of the upper estimate in~\eqref{theorem:optimalRates:statement}.
\qed

\section{Numerical experiments}\label{section:experiments}

The numerical experiments presented in this section were performed with the object-oriented \textsc{Matlab} software package MooAFEM from~\cite{ip2022}.
To ensure reproducibility, the source code together with all parameter settings is provided in the Code Ocean capsule~\cite{BGLP26codeocean}.

Algorithm~\ref{algorithm} has already been extensively investigated in the literature, both with the reconstruction estimator~\eqref{eq:ellipticReconstructionEstimator} in~\cite{cw2017, hw2020:convergence} and with the standard residual-based estimator~\eqref{eq:residualEstimator} in, e.g.,~\cite{hw2020:ailfem, hpsv2021, ghps2021, hpw2021}.
We therefore focus on complementary aspects such as the comparison between the reconstruction and standard estimators and the influence of the chosen scalar product in the Zarantonello iteration.
To examine these aspects independently of algebraic solver errors, the linear system~\eqref{algorithm:solve} in Algorithm~\ref{algorithm} is solved to machine precision with \textsc{Matlab}'s built-in direct solver \texttt{mldivide}.

\subsection{Benchmark 1}
\label{subsec:hpw21}
\makeatletter\edef\@currentlabel{\arabic{subsection}}\label{subsec:hpw21-short}\makeatother

We consider the nonlinearity $\mu(t)=1+e^{-t}$ on the Z-shaped domain $\Omega = (-1, 1)^2 \setminus\conv\{(0,0)^\intercal, (-1,0)^\intercal,(-1,-1)^\intercal\}$ and seek $\starred{u}\in H_0^1(\Omega)$ satisfying
\begin{equation*}
  -\div (\mu(|\nabla \starred{u}|^2) \, \nabla \starred{u})
  = \div\bigl(\chi_\omega (1,1)^\intercal\bigr)
  \;\:\text{in }\Omega,
\end{equation*}
where $\chi_\omega\colon\Omega\to\{0,1\}$ is the indicator function of $\omega\coloneq\conv\{(1,0)^\intercal, (1,1)^\intercal, (0,1)^\intercal\}$.
The nonlinearity $\mu$ fulfills the growth condition~\eqref{assumption:mu} with $\widetilde{\alpha}=1-2e^{-3/2}$ and $\widetilde{L}=2$.
We employ Algorithm~\ref{algorithm} with the standard $H_{0}^1(\Omega)$-scalar product $a_{H^1}(v,w)\coloneq\dual{\nabla v}{\nabla w}_{L^2(\Omega)}$ and the corresponding norm $\enorm{\,\cdot\,}=\norm{\nabla (\,\cdot \,)}_{L^2(\Omega)}$.
For this choice, Lemma~\ref{lemma:sharp-lipschitz-radial} in Appendix~\ref{appendix:sharpLipschitz} shows that the induced PDE operator $\AA$ from~\eqref{eq:nonlinearPDE} satisfies~\eqref{assumption:BM} even with the same constants $\alpha$ and $L$ as in~\eqref{assumption:mu}.
Moreover, we use $\delta=\alpha/L^2=(1-2e^{-3/2})/4 \approx 0.13843$, which yields the (theoretical) optimal contraction constant $\qctr$ in~\eqref{eq:zarantonelloContraction}.

To compare the reconstruction and residual-based estimator, we run Algorithm~\ref{algorithm} and the update-based variant of~\cite[Algorithm~2]{ghps2021}.
That is, we replace $\zeta_\ell(u_\ell^{k-1}; \starred{z}_\ell^k)$ in Algorithm~\ref{algorithm} by the standard residual-based estimator $\eta_\ell(u_\ell^k)$ from~\eqref{eq:residualEstimator}.
Figure~\ref{fig:Zshape_meshes} shows the mesh-size functions corresponding to the meshes generated by both variants on level $\ell=15$.
In both cases, refinement is concentrated near all three singularities, namely the one induced by the re-entrant corner and the two induced by the right-hand side at $(1,0)^\intercal$ and $(0,1)^\intercal$.
As predicted by Theorem~\ref{theorem:optimalRates} and~\cite[Theorem~8]{ghps2021}, respectively, Figure~\ref{fig:Zshape_indicatorRates:conv} exhibits the optimal convergence rate $-1/2$ for both variants with bulk parameter $\theta=0.5$.
In contrast, uniform refinement, corresponding to $\theta=1$, leads to suboptimal convergence rates.
The number of linearization iterations in Figure~\ref{fig:Zshape_indicatorRates:steps} is comparable.

\begin{figure}
  \centering
  \begin{subfigure}[t]{0.48\textwidth}
    \centering
    \begin{tikzpicture}
    \begin{axis}[%
        axis equal image,%
        width=\meshWidth,%
        xmin=-1.05, xmax=1.05,%
        ymin=-1.05, ymax=1.05,%
        font=\footnotesize%
    ]
        \addplot graphics [xmin=-1, xmax=1, ymin=-1, ymax=1]
        {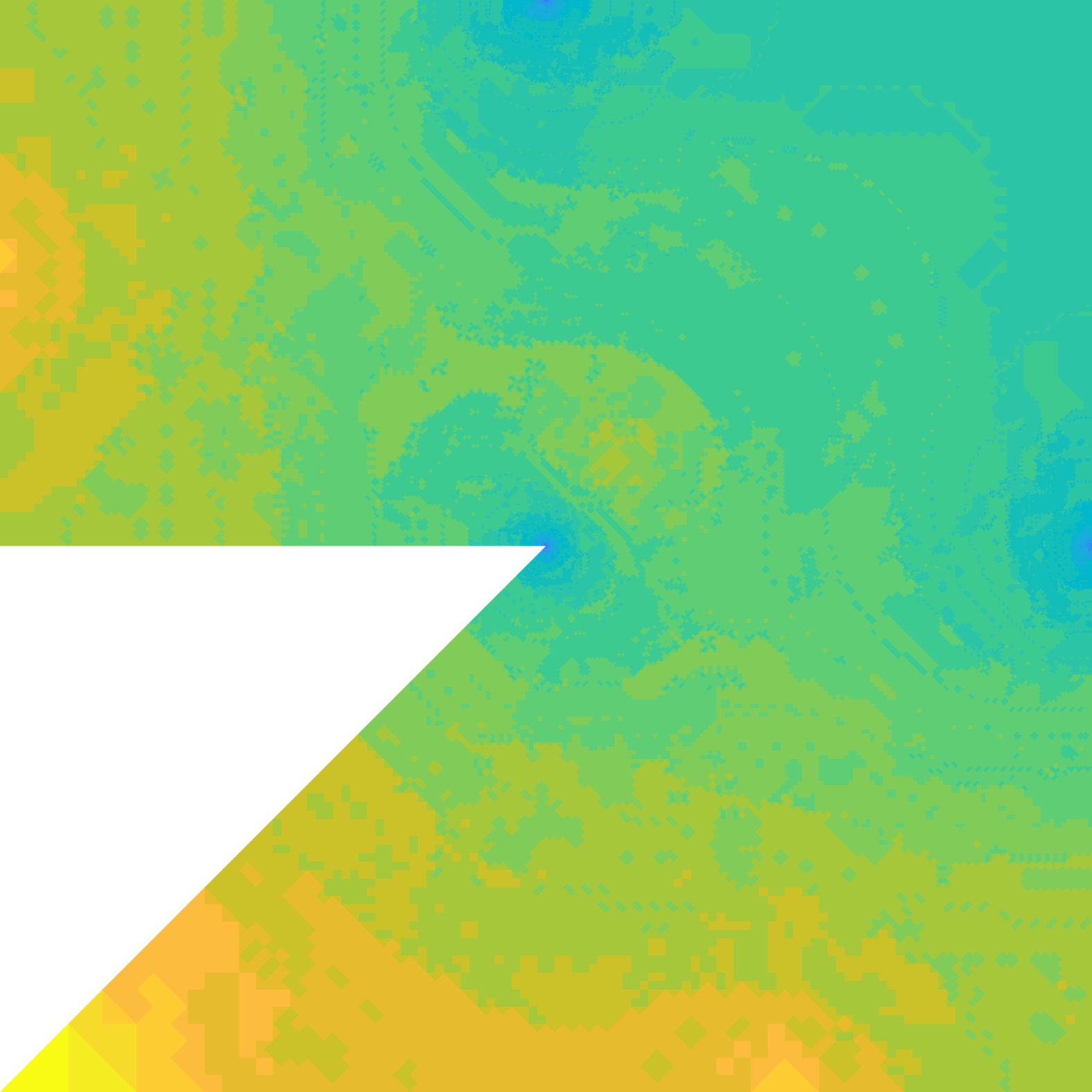};
    \end{axis}
\end{tikzpicture}
    \caption{$h_\ell$ generated with $\zeta$ for $\ell=15$, $\#\TT_{15}=318{,}871$}
    \label{fig:Zshape_meshes:reconstruction}
  \end{subfigure}
  \hfill
  \begin{subfigure}[t]{0.48\textwidth}
    \centering
    \begin{tikzpicture}
    \begin{axis}[%
        axis equal image,%
        width=\meshWidth,%
        xmin=-1.05, xmax=1.05,%
        ymin=-1.05, ymax=1.05,%
        font=\footnotesize,%
        point meta min=-4.66596493e+00,%
        point meta max=-1.05360498e+00,%
        colorbar,%
        colorbar style={%
            title={\(h_\ell\)},%
            font=\footnotesize,%
            width=2.5mm,%
            title style={yshift=-2mm},%
            yticklabel={$10^{\pgfmathprintnumber{\tick}}$},%
            % ytick={-1,-2,-3,-4,-5,-6},%
        },%
    ]
        \addplot graphics [xmin=-1, xmax=1, ymin=-1, ymax=1]
        {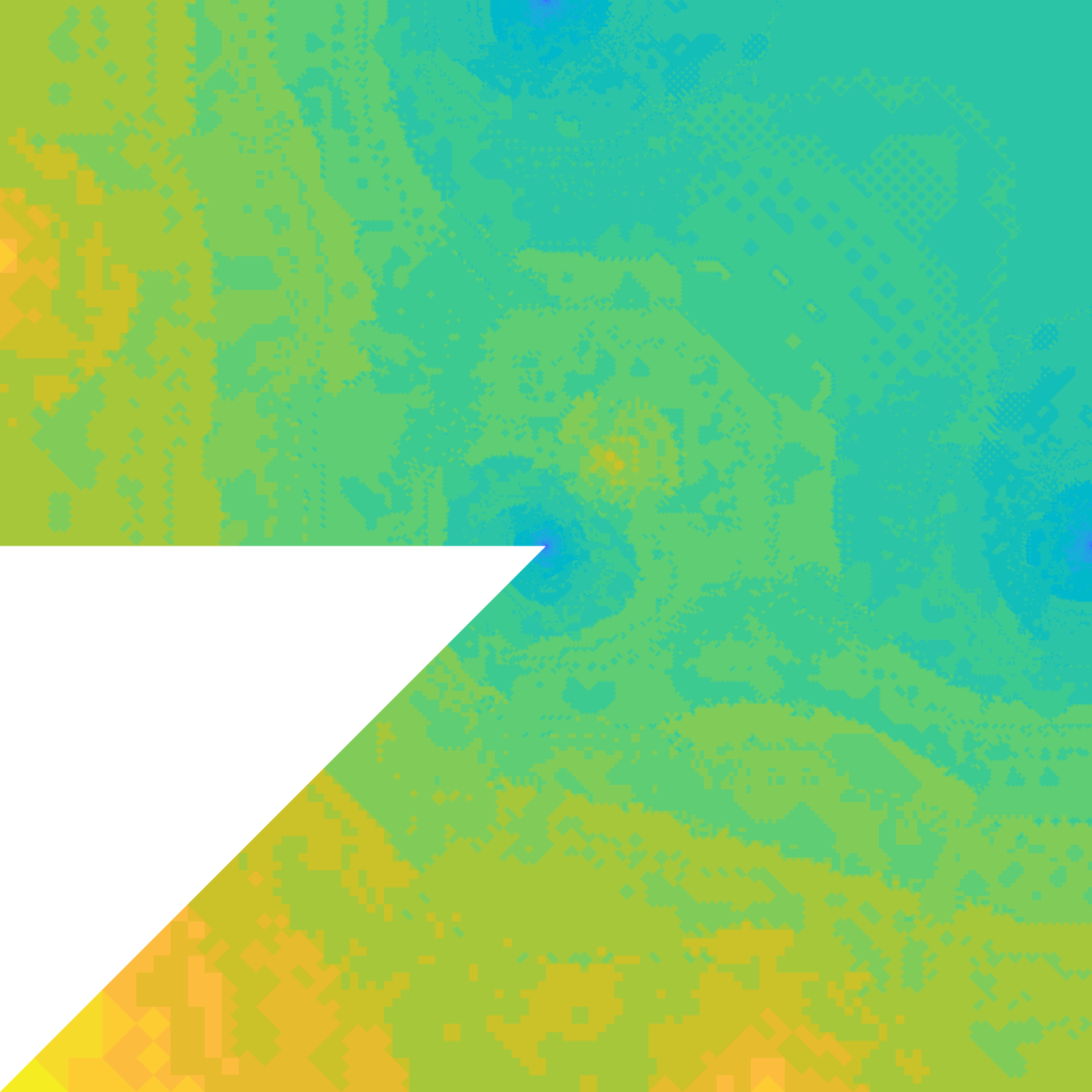};
    \end{axis}
\end{tikzpicture}
    \caption{$h_\ell$ generated with $\eta$ for $\ell=15$, $\#\TT_{15}=416{,}913$}
    \label{fig:Zshape_meshes:standard}
  \end{subfigure}
  \caption{Mesh-size functions $h_\ell$, given by $h_\ell|_T=|T|^{1/2}$ for all $T\in\TT_\ell$, for the level $\ell=15$ mesh generated by Algorithm~\ref{algorithm} for Benchmark~\ref{subsec:hpw21-short}.
    Refinement is steered either by the $\zeta_\ell(u_\ell^{k-1}; \starred{z}_\ell^k)$ from~\eqref{eq:ellipticReconstructionEstimator} (left) or by the standard estimator $\eta_\ell(u_\ell^k)$ from~\eqref{eq:residualEstimator} (right).
    The color bar applies to both plots.
    The parameters are set to $u_0^0=0, \theta=0.5, \protect\Cmark=1, \lctr=0.1, \delta=(1-2e^{-3/2})/4$, and $p=1$.
  }
  \label{fig:Zshape_meshes}
\end{figure}

\begin{figure}
  \centering
  \begin{subfigure}[t]{0.48\textwidth}
    \centering
    \vspace{0pt}
    \begin{tikzpicture}[>=stealth]

    %% READ DATA

    \pgfplotstableread[col sep=comma]{plots/Zshape_indicatorsRates/p1_theta50_lambda0.1_delta0.138435_nDofsMax10000000_scalarProductH1_estimatorTypereconstruction.csv}\reconstruction

    \pgfplotstableread[col sep=comma]{plots/Zshape_indicatorsRates/p1_theta50_lambda0.1_delta0.138435_nDofsMax10000000_scalarProductH1_estimatorTypestandard.csv}\standard

    \pgfplotstableread[col sep=comma]{plots/Zshape_indicatorsRates/p1_theta100_lambda0.1_delta0.138435_nDofsMax10000000_scalarProductH1_estimatorTypereconstruction.csv}\reconstructionUniform

    \pgfplotstableread[col sep=comma]{plots/Zshape_indicatorsRates/p1_theta100_lambda0.1_delta0.138435_nDofsMax10000000_scalarProductH1_estimatorTypestandard.csv}\standardUniform

    %% PLOT
    \begin{loglogaxis}[%
            width            = \convergenceWidth,%
            xlabel           = {$\dim\XX_\ell$},%
            xlabel style     = {text depth=2.4ex},%
            ylabel           = {estimator},%
            ymajorgrids      = true,%
            font             = \footnotesize,%
            grid style       = {%
                densely dotted,%
                semithick%
            },%
            legend style     = {%
                legend pos  = south west,%
                font = \scriptsize%
            },%
        ]

        \coordinate (legend) at (axis description cs:0.69,0.67);

        %% PLOT DATA

        \addplot+ [p1L2, iterative, forget plot]
        table [x={nDofs}, y={zeta}] {\reconstruction}; \label{zeta}

        \addplot+ [p1L2, iterative, forget plot, mark options={fill opacity=0}]
        table [x={nDofs}, y={zeta}] {\reconstructionUniform}; \label{zetaUniform}

        \addplot+ [p1L3, iterative, forget plot]
        table [x={nDofs}, y={eta}] {\standard}; \label{eta}

        \addplot+ [p1L3, iterative, forget plot, mark options={fill opacity=0}]
        table [x={nDofs}, y={eta}] {\standardUniform}; \label{etaUniform}
        
        %% SLOPE TRIANGLE
        \drawswappedslopetriangle{0.35}{9e6}{1.5e-2}
        \drawslopetriangle{0.5}{3e5}{1.6e-3}

    \end{loglogaxis}

    \matrix [
    matrix of nodes,
    anchor=south,
    font=\scriptsize,
    line join = round,
    line cap=round
    ] at (legend) {
         & $\theta=0.5$ & $\theta=1$ \\
        reconst. & \ref{zeta} & \ref{zetaUniform} \\
        standard & \ref{eta} & \ref{etaUniform}\\
    };
\end{tikzpicture}
    \caption{convergence history}
    \label{fig:Zshape_indicatorRates:conv}
  \end{subfigure}
  \hfill
  \begin{subfigure}[t]{0.48\textwidth}
    \centering
    \vspace{0pt}
    \begin{tikzpicture}[>=stealth]

    %% READ DATA

    \pgfplotstableread[col sep=comma]{plots/Zshape_indicatorsRates/p1_theta50_lambda0.1_delta0.138435_nDofsMax10000000_scalarProductH1_estimatorTypereconstruction.csv}\reconstruction

    \pgfplotstableread[col sep=comma]{plots/Zshape_indicatorsRates/p1_theta50_lambda0.1_delta0.138435_nDofsMax10000000_scalarProductH1_estimatorTypestandard.csv}\standard

    %% PLOT
    \begin{semilogxaxis}[%
            width            = \convergenceWidth,%
            xlabel           = {$\dim\XX_\ell$},%
            xlabel style     = {text depth=2.4ex},%
            ylabel           = {number of iterations},%
            ymajorgrids      = true,%
            font             = \footnotesize,%
            grid style       = {%
                densely dotted,%
                semithick%
            },%
            legend style     = {%
                legend pos  = north east,%
                font = \scriptsize%
            },%
        ]

        \coordinate (legend) at (axis description cs:0.74,0.67);

        %% PLOT DATA

        \addplot+ [p1L2, iterative, forget plot]
        table [x={nDofs}, y={ZarIterations}] {\reconstruction};

        \addplot+ [p1L3, iterative, forget plot]
        table [x={nDofs}, y={ZarIterations}] {\standard};

    \end{semilogxaxis}

    \matrix [
    matrix of nodes,
    anchor=south,
    font=\scriptsize,
    line join = round,
    line cap=round
    ] at (legend) {
         & $\theta=0.5$ \\
        reconst. & \ref{zeta} \\
        standard & \ref{eta}\\
    };
\end{tikzpicture}
    \caption{number $\kk[\ell]$ of linearization iterations}
    \label{fig:Zshape_indicatorRates:steps}
  \end{subfigure}
  \caption{Convergence history (left) and number of linearization iterations (right) of Algorithm~\ref{algorithm} over the number of degrees of freedom for Benchmark~\ref{subsec:hpw21-short}.
    The algorithm is steered by either the reconstruction estimator $\zeta_\ell(u_\ell^{k-1}; \starred{z}_\ell^k)$ from~\eqref{eq:ellipticReconstructionEstimator} (circle) or by the standard estimator $\eta_\ell(u_\ell^k)$ from~\eqref{eq:residualEstimator} (triangle).
    Solid markers correspond to the bulk parameter $\theta=0.5$ (adaptive mesh refinement), hollow markers to $\theta=1$ (uniform mesh refinement).
    The left plot displays the corresponding estimators $\zeta_\ell(u_\ell^{\kk-1}; \starred{z}_\ell^\kk)$ and $\eta_\ell(u_\ell^\kk)$.
    The parameters are set to $u_0^0=0, \protect\Cmark=1, \lctr=0.1, \delta=(1-2e^{-3/2})/4$, and $p=1$.
  }
  \label{fig:Zshape_indicatorRates}
\end{figure}
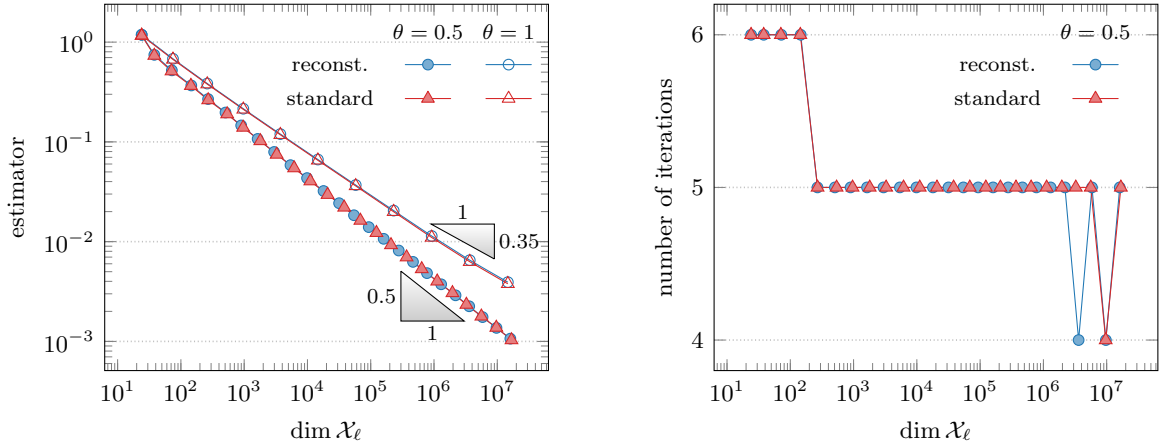

Finally, although stability~\eqref{axiom:stability} and estimator equivalence~\eqref{eq:estimatorEquivalence} are proven only for polynomial degree $p=1$, Figure~\ref{fig:Zshape_rates} shows the optimal rate $-p/2$ for Algorithm~\ref{algorithm} steered by the reconstruction estimator also for $p=2, 4$ with respect to the number of degrees of freedom, the overall computational cost from~\eqref{eq:costDef}, and the cumulative runtime.
Although the overall cost is not equivalent to the runtime here, due to the use of a direct solver for the update, optimal rates are observed in all three quantities.
Again, uniform refinement yields suboptimal convergence rates.

\begin{figure}
  \centering
  \begin{subfigure}[t]{0.48\textwidth}
    \centering
    \vspace{0pt}
    \begin{tikzpicture}[>=stealth]

    %% READ DATA

    \pgfplotstableread[col sep=comma]{plots/Zshape_rates/p2_theta50_lambda0.1_delta0.138435_nDofsMax10000000_scalarProductH1_estimatorTypereconstruction.csv}\Stwo

    \pgfplotstableread[col sep=comma]{plots/Zshape_rates/p4_theta50_lambda0.1_delta0.138435_nDofsMax10000000_scalarProductH1_estimatorTypereconstruction.csv}\Sfour

    \pgfplotstableread[col sep=comma]{plots/Zshape_rates/p2_theta100_lambda0.1_delta0.138435_nDofsMax10000000_scalarProductH1_estimatorTypereconstruction.csv}\StwoUniform

    \pgfplotstableread[col sep=comma]{plots/Zshape_rates/p4_theta100_lambda0.1_delta0.138435_nDofsMax10000000_scalarProductH1_estimatorTypereconstruction.csv}\SfourUniform

    %% PLOT
    \begin{loglogaxis}[%
            width            = \convergenceWidth,%
            xlabel           = {$\dim\XX_\ell$},%
            xlabel style     = {text depth=2.4ex},%
            ylabel           = {quasi-error},%
            ymajorgrids      = true,%
            font             = \footnotesize,%
            grid style       = {%
                densely dotted,%
                semithick%
            },%
            legend style     = {%
                legend pos  = south west,%
                font = \scriptsize%
            },%
        ]

        %% PLOT DATA

        \addplot+ [p2L2, iterative, forget plot]
        table [x={nDofs}, y={Zeta}] {\Stwo}; \label{Zetap2}

        \addplot+ [p2L2, iterative, forget plot, mark options={fill opacity=0}]
        table [x={nDofs}, y={Zeta}] {\StwoUniform}; \label{Zetap2Uniform}

        \addplot+ [p3L2, iterative, forget plot]
        table [x={nDofs}, y={Zeta}] {\Sfour}; \label{Zetap4}

        \addplot+ [p3L2, iterative, forget plot, mark options={fill opacity=0}]
        table [x={nDofs}, y={Zeta}] {\SfourUniform}; \label{Zetap4Uniform}
        
        %% SLOPE TRIANGLE
        \drawswappedslopetriangle{1}{6e6}{2.5e-4}
        \drawswappedslopetriangle{0.32}{7e6}{4e-2}
        \drawslopetriangle{2}{3.5e5}{1e-9}

    \end{loglogaxis}
\end{tikzpicture}
    \caption{number of degrees of freedom}
  \end{subfigure}
  \hfill
  \begin{subfigure}[t]{0.48\textwidth}
    \centering
    \vspace{0pt}
    \begin{tikzpicture}[>=stealth]

    %% READ DATA

    \pgfplotstableread[col sep=comma]{plots/Zshape_rates/p2_theta50_lambda0.1_delta0.138435_nDofsMax10000000_scalarProductH1_estimatorTypereconstruction.csv}\Stwo

    \pgfplotstableread[col sep=comma]{plots/Zshape_rates/p4_theta50_lambda0.1_delta0.138435_nDofsMax10000000_scalarProductH1_estimatorTypereconstruction.csv}\Sfour

    \pgfplotstableread[col sep=comma]{plots/Zshape_rates/p2_theta100_lambda0.1_delta0.138435_nDofsMax10000000_scalarProductH1_estimatorTypereconstruction.csv}\StwoUniform

    \pgfplotstableread[col sep=comma]{plots/Zshape_rates/p4_theta100_lambda0.1_delta0.138435_nDofsMax10000000_scalarProductH1_estimatorTypereconstruction.csv}\SfourUniform

    %% PLOT
    \begin{loglogaxis}[%
            width            = \convergenceWidth,%
            xlabel           = {$\sum_{\substack{(\ell^\prime, k^\prime)\in\QQ \\ \abs{\ell^\prime, k^\prime}\leq \abs{\ell, \kk}}}\dim\XX_{\ell^\prime}$},%
            xlabel style      = {text depth=2.4ex},%
            ylabel           = {quasi-error},%
            ymajorgrids      = true,%
            font             = \footnotesize,%
            grid style       = {%
                densely dotted,%
                semithick%
            },%
            legend style     = {%
                legend pos  = south west,%
                font = \scriptsize%
            },%
        ]

        %% PLOT DATA

        \addplot+ [p2L2, iterative, forget plot]
        table [x={work}, y={Zeta}] {\Stwo};

        \addplot+ [p2L2, iterative, forget plot, mark options={fill opacity=0}]
        table [x={work}, y={Zeta}] {\StwoUniform};

        \addplot+ [p3L2, iterative, forget plot]
        table [x={work}, y={Zeta}] {\Sfour};

        \addplot+ [p3L2, iterative, forget plot, mark options={fill opacity=0}]
        table [x={work}, y={Zeta}] {\SfourUniform};
        
        %% SLOPE TRIANGLE
        \drawswappedslopetriangle{1}{1e8}{2.5e-4}
        \drawswappedslopetriangle{0.32}{1e8}{3e-2}
        \drawslopetriangle{2}{9e6}{1e-9}

    \end{loglogaxis}
\end{tikzpicture}
    \caption{cumulative number of degrees of freedom}
  \end{subfigure}

  \bigskip
  \begin{subfigure}[b]{0.48\textwidth}
    \centering
    \begin{tikzpicture}[>=stealth]

    %% READ DATA

    \pgfplotstableread[col sep=comma]{plots/Zshape_rates/p2_theta50_lambda0.1_delta0.138435_nDofsMax10000000_scalarProductH1_estimatorTypereconstruction.csv}\Stwo

    \pgfplotstableread[col sep=comma]{plots/Zshape_rates/p4_theta50_lambda0.1_delta0.138435_nDofsMax10000000_scalarProductH1_estimatorTypereconstruction.csv}\Sfour

    \pgfplotstableread[col sep=comma]{plots/Zshape_rates/p2_theta100_lambda0.1_delta0.138435_nDofsMax10000000_scalarProductH1_estimatorTypereconstruction.csv}\StwoUniform

    \pgfplotstableread[col sep=comma]{plots/Zshape_rates/p4_theta100_lambda0.1_delta0.138435_nDofsMax10000000_scalarProductH1_estimatorTypereconstruction.csv}\SfourUniform

    %% PLOT
    \begin{loglogaxis}[%
            width            = \convergenceWidth,%
            xlabel           = {cumulative runtime [s]},%
            ylabel           = {quasi-error},%
            ymajorgrids      = true,%
            font             = \footnotesize,%
            grid style       = {%
                densely dotted,%
                semithick%
            },%
            legend style     = {%
                legend pos  = south west,%
                font = \scriptsize%
            },%
        ]

        %% PLOT DATA

        \addplot+ [p2L2, iterative, forget plot]
        table [x={runtime}, y={Zeta}] {\Stwo};

        \addplot+ [p2L2, iterative, forget plot, mark options={fill opacity=0}]
        table [x={runtime}, y={Zeta}] {\StwoUniform};

        \addplot+ [p3L2, iterative, forget plot]
        table [x={runtime}, y={Zeta}] {\Sfour};

        \addplot+ [p3L2, iterative, forget plot, mark options={fill opacity=0}]
        table [x={runtime}, y={Zeta}] {\SfourUniform};
        
        %% SLOPE TRIANGLE
        \drawswappedslopetriangle{0.32}{3e3}{3.5e-2}
        \drawswappedslopetriangle{1}{4e3}{1.5e-4}
        \drawslopetriangle{2}{1.9e2}{7e-10}

    \end{loglogaxis}
\end{tikzpicture}
    \caption{cumulative runtime}
  \end{subfigure}
  \hfill
  \begin{subfigure}[b]{0.48\textwidth}
    \centering
    \vspace{0pt}
    \begin{tikzpicture}[>=stealth]

  \matrix(m) [
    matrix of nodes,
    nodes in empty cells,
    anchor = center,
    font = \scriptsize,
    column 1/.style={anchor=base east},
  ] at (0,0) {
     &                                   &  &  &
     \\
     & $\theta=0.5$
     & $\theta=1$
    \\
    \(p=2\)
     & \ref*{Zetap2}
     & \ref*{Zetap2Uniform}
    \\
    \(p=4\)
     & \ref*{Zetap4}
     & \ref*{Zetap4Uniform}
    \\
  };

  \draw[line width=\arrayrulewidth, color=black] (m.south west |- m-2-1.south west) -- (m-2-3.south east);

\end{tikzpicture}
    \caption{legend}
  \end{subfigure}
  \caption{Convergence history of the computable quasi-error $\widetilde{\Zeta}_\ell^\kk$ from~\eqref{eq:quasiErrorComputable} of Algorithm~\ref{algorithm} for polynomial degrees $p=2$ (square) and $p=4$ (diamond) for Benchmark~\ref{subsec:hpw21-short}.
    Solid markers correspond to the bulk parameter $\theta=0.5$ (adaptive mesh refinement), hollow markers to $\theta=1$ (uniform mesh refinement).
    The parameters are set to $u_0^0=0, \protect\Cmark=1, \lctr=0.1$, and $\delta=(1-2e^{-3/2})/4$.
  }
  \label{fig:Zshape_rates}
\end{figure}

\subsection{Benchmark 2}
\label{subsec:mv23}
\makeatletter\edef\@currentlabel{\arabic{subsection}}\label{subsec:mv23-short}

For parameters $0<\tau<1$ and $q>1/2$, we consider the nonlinearity
\begin{equation}\label{eq:B2nonlinearity}
  \mu(t)=\frac{c_q+\tau}{1+c_q}+\frac{1-\tau}{1+c_q}(1+t)^{-q},
  \quad\text{ where }\quad
  c_q\coloneq 2\biggl(\frac{2q-1}{2(q+1)}\biggr)^{q+1}>0.
\end{equation}
Elementary calculus shows that the induced mapping $t\mapsto\mu(t^2)\,t$ satisfies
\begin{equation*}
  \tau
  \leq
  \frac{\d{}}{\d t}\bigl(\mu(t^2)\,t\bigr)
  \leq
  1
  \quad \text{ for all } t\geq 0.
\end{equation*}
Hence, $\mu$ satisfies the growth condition~\eqref{assumption:mu} with $\widetilde{\alpha}=\tau$ and $\widetilde{L}=1$.

For the benchmark, we take~\eqref{eq:PDEstrong} with mixed boundary conditions and the nonlinearity~\eqref{eq:B2nonlinearity} for $\tau=0.01$ and $q=11/20$.
On the L-shaped domain $\Omega=(-1,1)^2\setminus \bigl([0,1]\times [-1,0]\bigr)$, we impose Neumann boundary conditions on $\Gamma_N\coloneq \partial\Omega\cap\partial[-1,1]^2$ and homogeneous Dirichlet boundary conditions on $\Gamma_D\coloneq\partial\Omega\setminus\Gamma_N$.
We seek $\starred{u}\in H_D^1(\Omega)\coloneq\set{v\in H^1(\Omega) \given v|_{\Gamma_D}=0}$ satisfying
\begin{equation*}
  -\div (\mu(|\nabla \starred{u}|^2) \, \nabla \starred{u}) = f
  \;\:\text{in }\Omega\quad\text{ subject to }\quad
  \starred{u} = 0\;\:\text{on }\Gamma_D
  \;\:\text{ and }\;\:
  \mu(|\nabla \starred{u}|^2)\nabla \starred{u}\cdot\vec{n} =\phi\;\:\text{on }\Gamma_N.
\end{equation*}
The forcing term $f$ and the Neumann datum $\phi$ are chosen so that, in polar coordinates $(r, \varphi)$ with $\varphi\in(0,3\pi/2)$, the exact solution reads
\begin{equation*}
  \starred{u}(r, \varphi)=r^{2/3}\sin(2\varphi/3).
\end{equation*}
The choice $q=11/20>1/2$ ensures that $f\in L^2(\Omega)$.
Moreover, $\phi\in L^2(\Gamma_N)$.

We consider Algorithm~\ref{algorithm} steered by the reconstruction estimator $\zeta_\ell(u_\ell^{k-1}; \starred{z}_\ell^k)$ from~\eqref{eq:ellipticReconstructionEstimator} and compare the standard $H_{D}^1(\Omega)$-scalar product $a_{H^1}(v,w)=\dual{\nabla v}{\nabla w}_{L^2(\Omega)}$ with the (theoretical) $\mu$-weighted scalar product $a_\mu(v,w)\coloneq\dual{\mu(\abs{\nabla\starred{u}}^2)\,\nabla v}{\nabla w}_{L^2(\Omega)}$.
For $a_{H^1}(\,\cdot\,,\,\cdot\,)$, Lemma~\ref{lemma:sharp-lipschitz-radial} in Appendix~\ref{appendix:sharpLipschitz} shows that the induced operator $\AA$ from~\eqref{eq:nonlinearPDE} satisfies~\eqref{assumption:BM} on $\XX= H_D^1(\Omega)$ in the induced norm $\enorm{\,\cdot\,}_{H^1}=\norm{\nabla (\,\cdot \,)}_{L^2(\Omega)}$ with $\alpha=\widetilde{\alpha}=0.01$ and $L=\widetilde{L}=1$.
For $a_\mu(\,\cdot\,,\,\cdot\,)$, assumption~\eqref{assumption:BM} holds in the induced norm $\enorm{\,\cdot\,}_{\mu}=\norm{\mu(\abs{\nabla\starred{u}}^2)^{1/2}\,\nabla (\,\cdot \,)}_{L^2(\Omega)}$ with the admissible choice $\alpha=\widetilde{\alpha}/[\esssup_\Omega\mu(\abs{\nabla\starred{u}}^2)]$ and $L=\widetilde{L}/[\essinf_\Omega\mu(\abs{\nabla\starred{u}}^2)]$.
However, these global bounds are, in practice, too pessimistic.
Indeed, if $v,w\in\XX$ are sufficiently near $\starred{u}$ in the sense that $\mu(\abs{\nabla v}^2)\nabla v$ and $\mu(\abs{\nabla w}^2)\nabla w$ are close in $L^2(\Omega)$ to $\mu(\abs{\nabla\starred{u}}^2)\nabla v$ and $\mu(\abs{\nabla\starred{u}}^2)\nabla w$, respectively, then $\AA v-\AA w$ is well approximated by the operator induced by $a_\mu(\,\cdot\,,\,\cdot\,)$ applied to $v-w$.
In this local regime, the effective local constants associated with~\eqref{assumption:BM} are expected to be more favorable in the norm induced by $a_\mu(\,\cdot\,,\,\cdot\,)$.
This motivates the use of $a_\mu(\,\cdot\,,\,\cdot\,)$ in the Zarantonello iteration.
Clearly, this weighting is theoretical, since $\starred{u}$ is unknown (and sought) in practice.
However, replacing $\starred{u}$ by the current linearization point $u_\ell^{k-1}$ in Algorithm~\ref{algorithm} leads to the linearization-point-dependent Ka\v{c}anov scalar product $a_{(\ell,k)}(v,w)\coloneq \dual{\mu(\abs{\nabla u_\ell^{k-1}}^2)\,\nabla v}{\nabla w}_{L^2(\Omega)}$, which is a computable substitute for $a_\mu(\,\cdot\,,\,\cdot\,)$; cf.~Remark~\ref{rem:alternativeSchmes}.

Figure~\ref{fig:Lshape_comparison} compares the performance of Algorithm~\ref{algorithm} using either the $H_{D}^1(\Omega)$-scalar product $a_{H^1}(\,\cdot\,,\,\cdot\,)$ or the Ka\v{c}anov scalar product $a_{(\ell,k)}(\,\cdot\,,\,\cdot\,)$.
Computations with the $\mu$-weighted scalar product $a_\mu(\,\cdot\,,\,\cdot\,)$ were also performed, but the results are omitted because they are nearly indistinguishable from those obtained with $a_{(\ell,k)}(\,\cdot\,,\,\cdot\,)$.
For both scalar products, the figure additionally shows results obtained when the algorithm is driven with the standard estimator, i.e., when $\zeta_\ell(u_\ell^{k-1}; \starred{z}_\ell^k)$ is replaced in Algorithm~\ref{algorithm} by the standard residual-based estimator $\eta_\ell(u_\ell^k)$ from~\eqref{eq:residualEstimator}.
The two estimators yield nearly indistinguishable results.
All configurations attain the optimal convergence rate $-p/2$ in Figure~\ref{fig:Lshape_comparison:conv} and Figure~\ref{fig:Lshape_comparison:time}.
However, the Ka\v{c}anov iteration considerably reduces the computational cost by requiring fewer linearization iterations per mesh level; see Figure~\ref{fig:Lshape_comparison:steps}.

\begin{figure}
  \centering
  \begin{subfigure}[t]{0.48\textwidth}
    \centering
    \vspace{0pt}
    \begin{tikzpicture}[>=stealth]

  %% READ DATA

  \pgfplotstableread[col sep=comma]{plots/Lshape_comparison/p1_theta50_lambda0.01_delta1_tau0.01_nDofsMax10000000_tol0_scalarProductH1_estimatorTypereconstruction.csv}\PoneHoneR

  \pgfplotstableread[col sep=comma]{plots/Lshape_comparison/p1_theta50_lambda0.01_delta1_tau0.01_nDofsMax10000000_tol0_scalarProductH1_estimatorTypestandard.csv}\PoneHoneS

  \pgfplotstableread[col sep=comma]{plots/Lshape_comparison/p1_theta50_lambda0.01_delta1_tau0.01_nDofsMax10000000_tol0_scalarProductkacanov_estimatorTypereconstruction.csv}\PoneKacR

  \pgfplotstableread[col sep=comma]{plots/Lshape_comparison/p1_theta50_lambda0.01_delta1_tau0.01_nDofsMax10000000_tol0_scalarProductkacanov_estimatorTypestandard.csv}\PoneKacS

  \pgfplotstableread[col sep=comma]{plots/Lshape_comparison/p3_theta50_lambda0.01_delta1_tau0.01_nDofsMax10000000_tol0_scalarProductH1_estimatorTypereconstruction.csv}\PthreeHoneR

  \pgfplotstableread[col sep=comma]{plots/Lshape_comparison/p3_theta50_lambda0.01_delta1_tau0.01_nDofsMax10000000_tol0_scalarProductH1_estimatorTypestandard.csv}\PthreeHoneS

  \pgfplotstableread[col sep=comma]{plots/Lshape_comparison/p3_theta50_lambda0.01_delta1_tau0.01_nDofsMax10000000_tol0_scalarProductkacanov_estimatorTypereconstruction.csv}\PthreeKacR

  \pgfplotstableread[col sep=comma]{plots/Lshape_comparison/p3_theta50_lambda0.01_delta1_tau0.01_nDofsMax10000000_tol0_scalarProductkacanov_estimatorTypestandard.csv}\PthreeKacS

  %% PLOT
  \begin{loglogaxis}[%
      width            = \convergenceWidth,%
      xlabel           = {$\sum_{\substack{(\ell^\prime, k^\prime)\in\QQ \\ \abs{\ell^\prime, k^\prime}\leq \abs{\ell, \kk}}}\dim\XX_{\ell^\prime}$},%
      xlabel style     = {text depth=2.4ex},%
      ylabel           = {error},%
      ymajorgrids      = true,%
      xmax=9.9e9,%
      ymax=0.95,%
      font             = \footnotesize,%
      grid style       = {%
          densely dotted,%
          semithick%
        },%
      legend style     = {%
          legend pos  = south west,%
          font = \scriptsize%
        },%
    ]

    %% PLOT DATA

    \addplot+ [p1L2, iterative, forget plot, mark options={fill opacity=0}]
    table [x={work}, y={H1Error}] {\PoneHoneR}; \label{PoneHoneR}

    \addplot+ [p1L3, iterative, forget plot, mark options={fill opacity=0}]
    table [x={work}, y={H1Error}] {\PoneHoneS}; \label{PoneHoneS}

    \addplot+ [p1L2, iterative, forget plot]
    table [x={work}, y={H1Error}] {\PoneKacR}; \label{PoneKacR}

    \addplot+ [p1L3, iterative, forget plot]
    table [x={work}, y={H1Error}] {\PoneKacS}; \label{PoneKacS}

    \addplot+ [p2L3, iterative, forget plot, mark options={fill opacity=0}]
    table [x={work}, y={H1Error}] {\PthreeHoneR}; \label{PthreeHoneR}

    \addplot+ [p2L2, iterative, forget plot, mark options={fill opacity=0}]
    table [x={work}, y={H1Error}] {\PthreeHoneS}; \label{PthreeHoneS}

    \addplot+ [p2L3, iterative, forget plot]
    table [x={work}, y={H1Error}] {\PthreeKacR}; \label{PthreeKacR}

    \addplot+ [p2L2, iterative, forget plot]
    table [x={work}, y={H1Error}] {\PthreeKacS}; \label{PthreeKacS}

    %% SLOPE TRIANGLE
    \drawswappedslopetriangle{0.5}{3e8}{3e-3}
    \drawslopetriangle{1.5}{1e7}{1e-8}%Kacanov
    \drawswappedslopetriangle{1.5}{1.9e9}{9e-7}%Zarantonello

  \end{loglogaxis}
\end{tikzpicture}
    \caption{cumulative number of degrees of freedom}
    \label{fig:Lshape_comparison:conv}
  \end{subfigure}
  \hfill
      \begin{subfigure}[t]{0.48\textwidth}
    \centering
    \vspace{0pt}
    \begin{tikzpicture}[>=stealth]

  %% READ DATA

  \pgfplotstableread[col sep=comma]{plots/Lshape_comparison/p1_theta50_lambda0.01_delta1_tau0.01_nDofsMax10000000_tol0_scalarProductH1_estimatorTypereconstruction.csv}\PoneHoneR

  \pgfplotstableread[col sep=comma]{plots/Lshape_comparison/p1_theta50_lambda0.01_delta1_tau0.01_nDofsMax10000000_tol0_scalarProductH1_estimatorTypestandard.csv}\PoneHoneS

  \pgfplotstableread[col sep=comma]{plots/Lshape_comparison/p1_theta50_lambda0.01_delta1_tau0.01_nDofsMax10000000_tol0_scalarProductkacanov_estimatorTypereconstruction.csv}\PoneKacR

  \pgfplotstableread[col sep=comma]{plots/Lshape_comparison/p1_theta50_lambda0.01_delta1_tau0.01_nDofsMax10000000_tol0_scalarProductkacanov_estimatorTypestandard.csv}\PoneKacS

  \pgfplotstableread[col sep=comma]{plots/Lshape_comparison/p3_theta50_lambda0.01_delta1_tau0.01_nDofsMax10000000_tol0_scalarProductH1_estimatorTypereconstruction.csv}\PthreeHoneR

  \pgfplotstableread[col sep=comma]{plots/Lshape_comparison/p3_theta50_lambda0.01_delta1_tau0.01_nDofsMax10000000_tol0_scalarProductH1_estimatorTypestandard.csv}\PthreeHoneS

  \pgfplotstableread[col sep=comma]{plots/Lshape_comparison/p3_theta50_lambda0.01_delta1_tau0.01_nDofsMax10000000_tol0_scalarProductkacanov_estimatorTypereconstruction.csv}\PthreeKacR

  \pgfplotstableread[col sep=comma]{plots/Lshape_comparison/p3_theta50_lambda0.01_delta1_tau0.01_nDofsMax10000000_tol0_scalarProductkacanov_estimatorTypestandard.csv}\PthreeKacS

  %% PLOT
  \begin{loglogaxis}[%
      width            = \convergenceWidth,%
      xlabel           = {cumulative runtime [s]},%
      xlabel style     = {text depth=2.4ex},%
      ylabel           = {error},%
      ymajorgrids      = true,%
      ymax=0.95,%
      font             = \footnotesize,%
      grid style       = {%
          densely dotted,%
          semithick%
        },%
      legend style     = {%
          legend pos  = south west,%
          font = \scriptsize%
        },%
    ]

    %% PLOT DATA

    \addplot+ [p1L2, iterative, forget plot, mark options={fill opacity=0}]
    table [x={runtime}, y={H1Error}] {\PoneHoneR};

    \addplot+ [p1L3, iterative, forget plot, mark options={fill opacity=0}]
    table [x={runtime}, y={H1Error}] {\PoneHoneS};

    \addplot+ [p1L2, iterative, forget plot]
    table [x={runtime}, y={H1Error}] {\PoneKacR};

    \addplot+ [p1L3, iterative, forget plot]
    table [x={runtime}, y={H1Error}] {\PoneKacS};

    \addplot+ [p2L3, iterative, forget plot, mark options={fill opacity=0}]
    table [x={runtime}, y={H1Error}] {\PthreeHoneR};

    \addplot+ [p2L2, iterative, forget plot, mark options={fill opacity=0}]
    table [x={runtime}, y={H1Error}] {\PthreeHoneS};

    \addplot+ [p2L3, iterative, forget plot]
    table [x={runtime}, y={H1Error}] {\PthreeKacR};

    \addplot+ [p2L2, iterative, forget plot]
    table [x={runtime}, y={H1Error}] {\PthreeKacS};

    %% SLOPE TRIANGLE
    \drawswappedslopetriangle{0.5}{2e4}{3e-3}
    \drawslopetriangle{1.5}{5e2}{1e-8}%Kacanov
    \drawswappedslopetriangle{1.5}{7e4}{2e-6}%Zarantonello

  \end{loglogaxis}
\end{tikzpicture}
    \caption{cumulative runtime}
    \label{fig:Lshape_comparison:time}
  \end{subfigure}

  \bigskip
    \begin{subfigure}[b]{0.48\textwidth}
    \centering
    \vspace{0pt}
    \begin{tikzpicture}[>=stealth]

  %% READ DATA

  \pgfplotstableread[col sep=comma]{plots/Lshape_comparison/p1_theta50_lambda0.01_delta1_tau0.01_nDofsMax10000000_tol0_scalarProductH1_estimatorTypereconstruction.csv}\PoneHoneR

  \pgfplotstableread[col sep=comma]{plots/Lshape_comparison/p1_theta50_lambda0.01_delta1_tau0.01_nDofsMax10000000_tol0_scalarProductH1_estimatorTypestandard.csv}\PoneHoneS

  \pgfplotstableread[col sep=comma]{plots/Lshape_comparison/p1_theta50_lambda0.01_delta1_tau0.01_nDofsMax10000000_tol0_scalarProductkacanov_estimatorTypereconstruction.csv}\PoneKacR

  \pgfplotstableread[col sep=comma]{plots/Lshape_comparison/p1_theta50_lambda0.01_delta1_tau0.01_nDofsMax10000000_tol0_scalarProductkacanov_estimatorTypestandard.csv}\PoneKacS

  \pgfplotstableread[col sep=comma]{plots/Lshape_comparison/p3_theta50_lambda0.01_delta1_tau0.01_nDofsMax10000000_tol0_scalarProductH1_estimatorTypereconstruction.csv}\PthreeHoneR

  \pgfplotstableread[col sep=comma]{plots/Lshape_comparison/p3_theta50_lambda0.01_delta1_tau0.01_nDofsMax10000000_tol0_scalarProductH1_estimatorTypestandard.csv}\PthreeHoneS

  \pgfplotstableread[col sep=comma]{plots/Lshape_comparison/p3_theta50_lambda0.01_delta1_tau0.01_nDofsMax10000000_tol0_scalarProductkacanov_estimatorTypereconstruction.csv}\PthreeKacR

  \pgfplotstableread[col sep=comma]{plots/Lshape_comparison/p3_theta50_lambda0.01_delta1_tau0.01_nDofsMax10000000_tol0_scalarProductkacanov_estimatorTypestandard.csv}\PthreeKacS

  %% PLOT
  \begin{semilogxaxis}[%
      width            = \convergenceWidth,%
      xlabel           = {$\sum_{\substack{(\ell^\prime, k^\prime)\in\QQ \\ \abs{\ell^\prime, k^\prime}\leq \abs{\ell, \kk}}}\dim\XX_{\ell^\prime}$},%
      xlabel style     = {text depth=2.4ex},%
      ylabel           = {number of iterations},%
      ymajorgrids      = true,%
      font             = \footnotesize,%
      xmax             = 9.9e9,%
      grid style       = {%
          densely dotted,%
          semithick%
        },%
      legend style     = {%
          legend pos  = north east,%
          font = \scriptsize%
        },%
    ]

    %% PLOT DATA
    \addplot+ [p1L2, iterative, forget plot, mark options={fill opacity=0}]
    table [x={work}, y={ZarIterations}] {\PoneHoneR};

    \addplot+ [p1L3, iterative, forget plot, mark options={fill opacity=0}]
    table [x={work}, y={ZarIterations}] {\PoneHoneS};

    \addplot+ [p1L2, iterative, forget plot]
    table [x={work}, y={ZarIterations}] {\PoneKacR};

    \addplot+ [p1L3, iterative, forget plot]
    table [x={work}, y={ZarIterations}] {\PoneKacS};

    \addplot+ [p2L3, iterative, forget plot, mark options={fill opacity=0}]
    table [x={work}, y={ZarIterations}] {\PthreeHoneR};

    \addplot+ [p2L2, iterative, forget plot, mark options={fill opacity=0}]
    table [x={work}, y={ZarIterations}] {\PthreeHoneS};

    \addplot+ [p2L3, iterative, forget plot]
    table [x={work}, y={ZarIterations}] {\PthreeKacR};

    \addplot+ [p2L2, iterative, forget plot]
    table [x={work}, y={ZarIterations}] {\PthreeKacS};

  \end{semilogxaxis}
\end{tikzpicture}
    \caption{number $\kk[\ell]$ of linearization iterations}
    \label{fig:Lshape_comparison:steps}
  \end{subfigure}
  \hfill
  \begin{subfigure}[b]{0.48\textwidth}
    \centering
    \vspace{0pt}
    \begin{tikzpicture}[>=stealth]

  \matrix(m) [
    matrix of nodes,
    nodes in empty cells,
    anchor = center,
    font = \scriptsize,
    column 1/.style={anchor=base east},
  ] at (0,0) {
     &                                   &  &  &
    \\
     & $\zeta_\ell$
     & $\eta_\ell$
     & $\zeta_\ell$
     & $\eta_\ell$
    \\
    \(a_{H^1}(\,\cdot\,,\,\cdot\,)\)
     & \ref*{PoneHoneR}
     & \ref*{PoneHoneS}
     & \ref*{PthreeHoneR}
     & \ref*{PthreeHoneS}
    \\
    \(a_{(\ell,k)}(\,\cdot\,,\,\cdot\,)\)
     & \ref*{PoneKacR}
     & \ref*{PoneKacS}
     & \ref*{PthreeKacR}
     & \ref*{PthreeKacS}
    \\
  };

  \node[font = \scriptsize\bfseries, anchor=base] at ($(m-1-2)!0.5!(m-1-3)$) {$p=1$};
  \node[font = \scriptsize\bfseries, anchor=base] at ($(m-1-4)!0.5!(m-1-5)$) {$p=3$};
  \draw[line width=\arrayrulewidth, color=black] (m.south west |- m-2-1.south west) -- (m-2-5.south east);
  \draw[line width=\arrayrulewidth]
  ([xshift=-6pt]m-1-2.south west) -- ([xshift=6pt]m-1-3.south east);
  \draw[line width=\arrayrulewidth]
  ([xshift=-6pt]m-1-4.south west) -- ([xshift=6pt]m-1-5.south east);

\end{tikzpicture}
    \caption{legend}
    \label{fig:Lshape_comparison:legend}
  \end{subfigure}
  \caption{Convergence history of the error
    $\protect\norm{\nabla(\starred{u}-u_\ell^\kk)}_{L^2(\Omega)}$ and the number of linearization iterations for the scalar products $a_{H^1}(\,\cdot\,,\,\cdot\,)$ and $a_{(\ell,k)}(\,\cdot\,,\,\cdot\,)$ and various polynomial degrees~$p$ for Benchmark~\ref{subsec:mv23-short}.
    The algorithm is steered by either the reconstruction estimator $\zeta_\ell(u_\ell^{k-1}; \starred{z}_\ell^k)$ from~\eqref{eq:ellipticReconstructionEstimator} or by the standard estimator $\eta_\ell(u_\ell^k)$ from~\eqref{eq:residualEstimator}.
    The parameters are set to $u_0^0=0, \theta=0.5, \protect\Cmark=1, \lctr=0.01$, and $\delta=1$.
  }
  \label{fig:Lshape_comparison}
\end{figure}

With $\cost(\ell,k)$ from~\eqref{eq:costDef}, Table~\ref{table:parameterStudy} reports the error-weighted cumulative computational cost
\begin{equation}\label{eq:weightedCost}
  \norm{\nabla(\starred{u}-u_\ell^k)}_{L^2(\Omega)}\cost(\ell,k)^{1/2}
\end{equation}
together with the average number of linearization iterations $\kk[\ell]$ on the last three mesh levels for the three scalar products $a_{H^1}(\,\cdot\,,\,\cdot\,)$, $a_\mu(\,\cdot\,,\,\cdot\,)$, and $a_{(\ell,k)}(\,\cdot\,,\,\cdot\,)$ for different values of $\lambda$ and $\delta$.
For all scalar products, the weighted computational cost decreases as $\delta$ increases, although the smallest tested value, $\delta=0.1$, already exceeds the theoretical upper bound ensuring contraction for $a_{H^1}(\,\cdot\,,\,\cdot\,)$, namely $\delta=2\alpha/L^2=0.02$.
Furthermore, larger values of $\delta$ and $\lambda$ lead to a smaller number of linearization iterations $\kk[\ell]$.
Moreover, for every tested pair of parameters $\lambda$ and $\delta$, the $\mu$-weighted scalar products $a_\mu(\,\cdot\,,\,\cdot\,)$ and $a_{(\ell,k)}(\,\cdot\,,\,\cdot\,)$ yield smaller weighted costs than $a_{H^1}(\,\cdot\,,\,\cdot\,)$.
The performance of $a_\mu(\,\cdot\,,\,\cdot\,)$ and $a_{(\ell,k)}(\,\cdot\,,\,\cdot\,)$ is comparable.
For $a_{H^1}(\,\cdot\,,\,\cdot\,)$, the best results are obtained with $\lambda=0.05$ for all tested values of $\delta$.
In contrast, for the $\mu$-weighted scalar products, larger values of $\lambda$ are preferable, especially once $\delta$ is sufficiently large.
\begin{table}
  \centering
  {\footnotesize
\begin{booktabs}{
  colspec={c|cccc|cccc|cccc},
}
\toprule
scalar product & \SetCell[c=4]{c} $a_{H^1}(\,\cdot\,,\,\cdot\,)$ &&&& \SetCell[c=4]{c} $a_\mu(\,\cdot\,,\,\cdot\,)$ &&&& \SetCell[c=4]{c} $a_{(\ell,k)}(\,\cdot\,,\,\cdot\,)$\\
\midrule
\diagbox[height=1.5\line]{$\lambda$}{$\delta$} & $0.1$ & $0.5$ & $1$ & $1.5$ & $0.1$ & $0.5$ & $1$ & $1.5$ & $0.1$ & $0.5$ & $1$ & $1.5$\\
\midrule
& \SetCell[c=12]{c} error-weighted cumulative computational cost \\
\midrule

0.01 & 19.79 & 8.87 & 6.32 & 5.25 & 12.80 & 5.73 & 4.08 & \hb{3.54} & 12.80 & 5.75 & 4.10 & \hb{3.54} \\
0.05 & \hy{16.87} & \hy{7.48} & \hy{5.27} & \hy{4.22} & 9.19 & 4.18 & 2.93 & \hb{2.81} & 9.35 & 4.21 & 2.95 & 2.82 \\
0.10 & 19.33 & 8.33 & 5.80 & 4.71 & \hy{8.38} & \hy{3.75} & 2.62 & \hb{2.37} & \hy{9.08} & \hy{3.88} & \hy{2.64} & 2.39 \\
0.50 & 58.43 & 25.52 & 18.37 & 8.95 & 12.37 & 4.96 & \hy{2.49} & \hg{1.94} & 16.45 & 5.84 & 3.14 & \hy{2.09} \\
1.00 & 87.93 & 40.48 & 36.14 & 8.95 & 18.73 & 7.49 & \hy{2.49} & \hb{2.04} & 25.19 & 10.87 & 3.14 & \hy{2.09} \\
\midrule
& \SetCell[c=12]{c} average number of linearization iterations $\kk[\ell]$ on the last three levels \\
\midrule
0.01 & 131.33 & 26.67 & 13.67 & 9.33 & 62.00 & 12.67 & 6.33 & 5.33 & 62.33 & 12.67 & 6.33 & 5.33 \\
0.05 & 44.00 & 9.00 & 5.00 & 3.00 & 27.67 & 6.00 & 3.00 & 3.00 & 28.00 & 6.00 & 3.00 & 3.00 \\
0.10 & 28.33 & 6.00 & 3.00 & 2.00 & 18.00 & 4.00 & 2.00 & 2.00 & 18.00 & 4.00 & 2.00 & 2.00 \\
0.50 & 16.67 & 3.33 & 2.00 & 1.00 & 6.67 & 1.33 & 1.00 & 1.00 & 7.00 & 1.33 & 1.00 & 1.00 \\
1.00 & 13.33 & 3.00 & 1.33 & 1.00 & 5.67 & 1.00 & 1.00 & 1.00 & 6.00 & 1.00 & 1.00 & 1.00 \\
\bottomrule
\end{booktabs}
}
  \caption{Error-weighted cumulative computational cost from~\eqref{eq:weightedCost} and average number of linearization iterations $\kk[\ell]$ on the last three mesh levels for Benchmark~\ref{subsec:mv23-short}.
    The results are shown for various choices of $\lambda$ and $\delta$, where the algorithm is terminated once $\norm{\nabla(\starred{u}-u_\ell^\kk)}_{L^2(\Omega)}\le 10^{-2}$ is reached.
    The smallest row-wise and column-wise values of the error-weighted cost are highlighted in blue and yellow, respectively, and in green if both coincide.
    The parameters are set to $u_0^0=0, \theta=0.5, \protect\Cmark=1$, and $p=1$.
  }
  \label{table:parameterStudy}
\end{table}

Finally, Figure~\ref{fig:Lshape_rates} illustrates the results for $\lambda=0.1$ and $\delta=1.5$, which is close to the best-performing choices for all three considered scalar products in Table~\ref{table:parameterStudy}.
For $a_\mu(\,\cdot\,,\,\cdot\,)$ as well as $a_{(\ell,k)}(\,\cdot\,,\,\cdot\,)$, Algorithm~\ref{algorithm} achieves the optimal convergence rate $-1/2$ for the error in Figure~\ref{fig:Lshape_rates:conv}.
Moreover, Figure~\ref{fig:Lshape_rates:steps} shows that only two linearization iterations are performed on each mesh except for the initial one.
In contrast, for $a_{H^1}(\,\cdot\,,\,\cdot\,)$, the same choice $\lambda=0.1$ and $\delta=1.5$ leads to a suboptimal convergence rate of the error in Figure~\ref{fig:Lshape_rates:conv}.
This behavior suggests that, for the choice $\delta=1.5$, the Zarantonello iteration with $a_{H^1}(\,\cdot\,,\,\cdot\,)$ no longer provides a strong contraction $\qctr=\qctr[\delta]$.
Consequently, a smaller parameter $\lctr=\lctr[\qctr]$ may be required; cf.
\ the smallness condition~\eqref{theorem:optimalRates:assumptions} in Theorem~\ref{theorem:optimalRates}.
Undisplayed numerical experiments show optimal convergence rates for $a_{H^1}(\,\cdot\,,\,\cdot\,)$ with $\delta=1.5$ when a smaller value of $\lambda$ is used.
Overall, the results indicate that using a $\mu$-weighted scalar product like $a_\mu(\,\cdot\,,\,\cdot\,)$ or $a_{(\ell,k)}(\,\cdot\,,\,\cdot\,)$ in the Zarantonello iteration~\eqref{eq:zarantonello} improves the performance of the algorithm.
\begin{figure}
  \centering
  \begin{subfigure}[t]{0.48\textwidth}
    \centering
    \vspace{0pt}
    \begin{tikzpicture}[>=stealth]

    %% READ DATA

    \pgfplotstableread[col sep=comma]{plots/Lshape_rates/p1_theta50_lambda0.1_delta1.5_tau0.01_nDofsMax10000000_tol0_scalarProductH1_estimatorTypereconstruction.csv}\Hone

    \pgfplotstableread[col sep=comma]{plots/Lshape_rates/p1_theta50_lambda0.1_delta1.5_tau0.01_nDofsMax10000000_tol0_scalarProductweighted_estimatorTypereconstruction.csv}\weighted

    \pgfplotstableread[col sep=comma]{plots/Lshape_rates/p1_theta50_lambda0.1_delta1.5_tau0.01_nDofsMax10000000_tol0_scalarProductkacanov_estimatorTypereconstruction.csv}\kacanov

    %% PLOT
    \begin{loglogaxis}[%
            width            = \convergenceWidth,%
            xlabel           = {$\sum_{\substack{(\ell^\prime, k^\prime)\in\QQ \\ \abs{\ell^\prime, k^\prime}\leq \abs{\ell, \kk}}}\dim\XX_{\ell^\prime}$},%
            xlabel style     = {text depth=2.4ex},%
            ylabel           = {error},%
            ymajorgrids      = true,%
            font             = \footnotesize,%
            grid style       = {%
                densely dotted,%
                semithick%
            },%
            legend style     = {%
                legend pos  = south west,%
                font = \scriptsize%
            },%
        ]

        \coordinate (legend) at (axis description cs:0.26,0.01);

        %% PLOT DATA
        
        \addplot+ [p1L2, iterative, forget plot]
        table [x={work}, y={H1Error}] {\Hone}; \label{errorH1}

        \addplot+ [p1L3, iterative, forget plot]
        table [x={work}, y={H1Error}] {\weighted}; \label{errorWeighted}

        \addplot+ [p2L2, iterative, forget plot]
        table [x={work}, y={H1Error}] {\kacanov}; \label{errorKacanov}
        
        %% SLOPE TRIANGLE
        \drawswappedslopetriangle{0.26}{8.2e7}{5e-3}
        \drawslopetriangle{0.5}{2e6}{4.2e-4}

    \end{loglogaxis}

    \matrix [
    matrix of nodes,
    anchor=south,
    font=\scriptsize,
    line join = round,
    line cap=round
    ] at (legend) {
        \ref{errorH1} & $a_{H^1}(\,\cdot\,,\,\cdot\,)$ \\
        \ref{errorWeighted} & $a_{\mu}(\,\cdot\,,\,\cdot\,)$ \\
        \ref{errorKacanov} & $a_{(\ell, k)}(\,\cdot\,,\,\cdot\,)$ \\
    };
\end{tikzpicture}
    \caption{convergence history}
    \label{fig:Lshape_rates:conv}
  \end{subfigure}
  \hfill
  \begin{subfigure}[t]{0.48\textwidth}
    \centering
    \vspace{0pt}
    \begin{tikzpicture}[>=stealth]

    %% READ DATA

    \pgfplotstableread[col sep=comma]{plots/Lshape_rates/p1_theta50_lambda0.1_delta1.5_tau0.01_nDofsMax10000000_tol0_scalarProductH1_estimatorTypereconstruction.csv}\Hone

    \pgfplotstableread[col sep=comma]{plots/Lshape_rates/p1_theta50_lambda0.1_delta1.5_tau0.01_nDofsMax10000000_tol0_scalarProductweighted_estimatorTypereconstruction.csv}\weighted

    \pgfplotstableread[col sep=comma]{plots/Lshape_rates/p1_theta50_lambda0.1_delta1.5_tau0.01_nDofsMax10000000_tol0_scalarProductkacanov_estimatorTypereconstruction.csv}\kacanov

    %% PLOT
    \begin{semilogxaxis}[%
            width            = \convergenceWidth,%
            xlabel           = {$\sum_{\substack{(\ell^\prime, k^\prime)\in\QQ \\ \abs{\ell^\prime, k^\prime}\leq \abs{\ell, \kk}}}\dim\XX_{\ell^\prime}$},%
            xlabel style     = {text depth=2.4ex},%
            ylabel           = {number of iterations},%
            ymajorgrids      = true,%
            font             = \footnotesize,%
            grid style       = {%
                densely dotted,%
                semithick%
            },%
            legend style     = {%
                legend pos  = north east,%
                font = \scriptsize%
            },%
        ]

        \coordinate (legend) at (axis description cs:0.48,0.58);

        %% PLOT DATA

        \addplot+ [p1L2, iterative, forget plot]
        table [x={work}, y={ZarIterations}] {\Hone}; \label{H1}

        \addplot+ [p1L3, iterative, forget plot]
        table [x={work}, y={ZarIterations}] {\weighted}; \label{weighted}

        \addplot+ [p2L2, iterative, forget plot]
        table [x={work}, y={ZarIterations}] {\kacanov}; \label{kacanov}

    \end{semilogxaxis}

    \matrix [
    matrix of nodes,
    anchor=south,
    font=\scriptsize,
    line join = round,
    line cap=round
    ] at (legend) {
        \ref{errorH1} & $a_{H^1}(\,\cdot\,,\,\cdot\,)$ \\
        \ref{errorWeighted} & $a_{\mu}(\,\cdot\,,\,\cdot\,)$ \\
        \ref{errorKacanov} & $a_{(\ell, k)}(\,\cdot\,,\,\cdot\,)$ \\
    };
\end{tikzpicture}
    \caption{number $\kk[\ell]$ of linearization iterations}
    \label{fig:Lshape_rates:steps}
  \end{subfigure}
  \caption{Convergence history (left) of the error
    $\protect\norm{\nabla(\starred{u}-u_\ell^\kk)}_{L^2(\Omega)}$ and the number of linearization iterations (right) for the scalar products $a_{H^1}(\,\cdot\,,\,\cdot\,)$ (circle), $a_\mu(\,\cdot\,,\,\cdot\,)$ (triangle), and $a_{(\ell,k)}(\,\cdot\,,\,\cdot\,)$ (square) for Benchmark~\ref{subsec:mv23-short}.
    The parameters are set to $u_0^0=0, \theta=0.5, \protect\Cmark=1, \lctr=0.1, \delta=1.5$, and $p=1$.
  }
  \label{fig:Lshape_rates}
\end{figure}
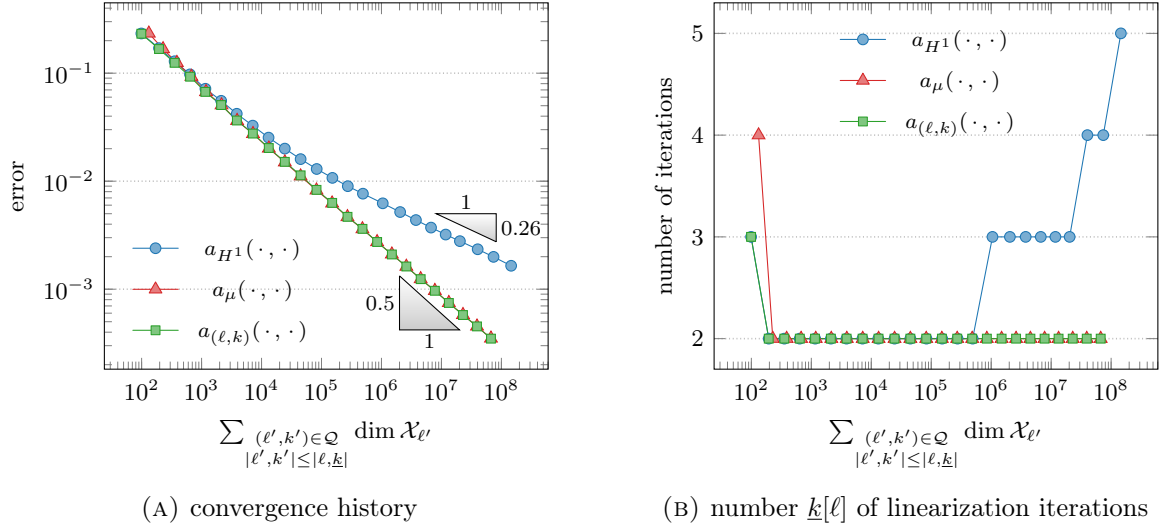

% ------------------------------------------------------------------------------
%% Bibliography
% ------------------------------------------------------------------------------
{
\sloppy
\printbibliography
}

% ------------------------------------------------------------------------------
%% Appendices
% ------------------------------------------------------------------------------

\appendix
\section{Sharp monotonicity and Lipschitz constants}
\label{appendix:sharpLipschitz}

Under the growth condition~\eqref{assumption:mu},~\cite[Lemma~25.26]{zeidler} shows that the operator $\AA$ from~\eqref{eq:nonlinearPDE} satisfies~\eqref{assumption:BM} with $\alpha=\widetilde{\alpha}$ and the non-sharp Lipschitz constant $L=3\widetilde{L}$ with respect to $\enorm{\,\cdot\,}=\norm{\nabla (\,\cdot \,)}_{L^2(\Omega)}$.
The following lemma improves the Lipschitz constant to the sharp value $L=\widetilde{L}$.

\begin{lemma}\label{lemma:sharp-lipschitz-radial}
  Let $\mu\colon[0, +\infty)\to [0, +\infty)$ satisfy, for constants $0<\widetilde{\alpha}\leq \widetilde{L}$, the growth condition
  \begin{equation}\label{eq:GCmu}
    \widetilde{\alpha} \, (t - s) \leq \mu(t^2) t - \mu(s^2) s \leq \widetilde{L} \, (t - s)
    \quad \text{for all } 0 \leq s \leq t.
  \end{equation}
  Then, $F\colon\R^d\to\R^d$, $F(\xi)\coloneq\mu(|\xi|^2)\xi$
  satisfies
  \begin{equation}\label{eq:F_BM}
    \widetilde{\alpha}\,\abs{\xi-\chi}^2\leq\bigl[F(\xi)-F(\chi)\bigr]\cdot(\xi-\chi)
    \quad\text{ and }\quad
    \abs{F(\xi)-F(\chi)}\leq \widetilde L\,\abs{\xi-\chi}
    \quad\text{ for all }\xi,\chi\in\mathbb R^d.
  \end{equation}
  Consequently, the induced PDE operator $\AA\colon H_0^1(\Omega)\to H^{-1}(\Omega)$,
  $\dual{\AA u}{v}\coloneq\dual{F(\nabla u)}{\nabla v}_{L^2(\Omega)}$ satisfies, for all $u,v,w\in H_0^1(\Omega)$,
  \begin{equation*}
    \widetilde{\alpha} \,
    \norm{\nabla(v-w)}_{L^2(\Omega)}^2
    \leq
    \dual{\AA v - \AA w}{v - w}
    \quad\text{ and }\quad
    \dual{\AA v - \AA w}{u}
    \le \widetilde{L} \, \norm{\nabla(v-w)}_{L^2(\Omega)} \, \norm{\nabla u}_{L^2(\Omega)}.
  \end{equation*}
\end{lemma}
\begin{proof}
  The monotonicity estimate in~\eqref{eq:F_BM} follows from~\cite[Lemma~25.26]{zeidler}.
  To establish the Lipschitz estimate, set $g(t)\coloneq \mu(t^2)t$.
  Then,~\eqref{eq:GCmu} gives
  \begin{equation}\label{eq:GCg}
    \abs{g(s)-g(t)} \le \widetilde{L}\,\abs{s-t}
    \quad\text{ and }\quad
    0\leq g(t)\leq\widetilde{L}\,t
    \quad \text{for all } s,t \geq 0.
  \end{equation}
  For $\xi,\chi\in\R^d$, write $\xi=s\sigma,\chi=t\rho$ with $s,t\ge0$ and $|\sigma|=|\rho|=1$.
  Using the identity
  \begin{equation}\label{eq:GCaux}
    |\sigma-\rho|^2=2-2\sigma\cdot \rho
  \end{equation}
  together with~\eqref{eq:GCg}, it follows that
  \begin{equation*}\begin{aligned}
       & |F(\xi)-F(\chi)|^2
      = |g(s)\sigma-g(t)\rho|^2
      = g(s)^2-2g(s)g(t)\sigma\cdot \rho+g(t)^2                                                                                                     \\
       & \eqreff{eq:GCaux}{=} |g(s)-g(t)|^2 + g(s)g(t)|\sigma-\rho|^2 \eqreff{eq:GCg}{\leq} \widetilde{L}^2 \bigl[|s-t|^2+ st |\sigma-\rho|^2\bigr]
      \eqreff{eq:GCaux}{=}\widetilde{L}^2 |s\sigma-t\rho|^2
      = \widetilde{L}^2 |\xi-\chi|^2 .
    \end{aligned}
  \end{equation*}
  Thus, $F$ is Lipschitz continuous with constant $\widetilde L$.
  The estimates for $\AA$ follow from~\eqref{eq:F_BM} by integration and the Cauchy--Schwarz inequality.
  This concludes the proof.
\end{proof}

\section{Proof of the Dörfler criterion~\texorpdfstring{\eqref{lemma:doerflerSufficient:eq1}}{(31)}}
\label{subsec:DoerflerSufficient}

\begin{proof}[Proof of~\eqref{lemma:doerflerSufficient:eq1}]
  The Dörfler criterion~\eqref{algorithm:doerfler} for $\zeta_\ell$, equivalence~\eqref{eq:estimatorEquivalence}, and the stopping criterion~\eqref{algorithm:stoppingCriterion} yield
  \begin{equation*}\begin{aligned}
      \theta^{1/2} \zeta_\ell(u_{\ell}^{\underline k -1}; \starred{z}_\ell^{\underline k})
       & \eqreff*{algorithm:doerfler}{\:\leq\:}\zeta_\ell(u_{\ell}^{\underline k -1};\MM_\ell, \starred{z}_\ell^{\underline k})     \\
       & \eqreff*{eq:estimatorEquivalence}{\:\leq\:}\eta_\ell(\MM_\ell, u_\ell^\kk) + (1+\delta)\Ceqv\,\enorm{\starred{z}_\ell^\kk}
      \eqreff{algorithm:stoppingCriterion}{\leq}\eta_\ell(\MM_\ell, u_\ell^\kk) + (1+\delta)\Ceqv\lctr\,\zeta_\ell(u_{\ell}^{\underline k -1}; \starred{z}_\ell^{\underline k}).
    \end{aligned}\end{equation*}
  This implies that
  \begin{equation}\label{eq:doerflerSufficient:auxiliary1}
    \big[\theta^{1/2} - (1+\delta)\Ceqv\lctr\big] \, \zeta_\ell(u_{\ell}^{\underline k -1}; \starred{z}_\ell^{\underline k})
    \le\eta_\ell(\MM_\ell, u_\ell^{\underline k}).
  \end{equation}
  The estimator equivalence~\eqref{eq:estimatorEquivalence} and the stopping criterion~\eqref{algorithm:stoppingCriterion} further verify
  \begin{equation}\label{eq:doerflerIsSufficient:auxiliary2}
    \begin{aligned}
      \eta_\ell(u_\ell^{\underline k})
      \eqreff{eq:estimatorEquivalence}\le
      \zeta_\ell(u_{\ell}^{\underline k -1}; \starred{z}_\ell^{\underline k}) + (1+\delta)\Ceqv \, \enorm{\starred{z}_{\ell}^{\underline k}}
      \eqreff{algorithm:stoppingCriterion}\le
      \big[1 + (1+\delta)\Ceqv \lctr\big] \, \zeta_\ell(u_{\ell}^{\underline k -1}; \starred{z}_\ell^{\underline k}).
    \end{aligned}
  \end{equation}
  The combination of~\eqref{eq:doerflerSufficient:auxiliary1}--\eqref{eq:doerflerIsSufficient:auxiliary2} concludes the proof of~\eqref{lemma:doerflerSufficient:eq1}.
\end{proof}

\end{document}